\documentclass[12pt]{article}
\input epsf.tex

\usepackage{graphicx}
\usepackage{amsmath,amsthm,amsfonts,amscd,amssymb,comment,eucal,latexsym,mathrsfs}
\usepackage{stmaryrd}
\usepackage[all]{xy}

\usepackage{epsfig}

\usepackage[all]{xy}
\xyoption{poly}
\usepackage{fancyhdr}
\usepackage{wrapfig}
\usepackage{epsfig}

\theoremstyle{plain}
\newtheorem{thm}{Theorem}[section]
\newtheorem{prop}[thm]{Proposition}
\newtheorem{lem}[thm]{Lemma}
\newtheorem{cor}[thm]{Corollary}

\newtheorem{claim}{Claim}
\theoremstyle{definition}
\newtheorem{defn}{Definition}
\theoremstyle{remark}
\newtheorem{remark}{Remark}

\newtheorem{notation}{Notation}

\advance \topmargin by -\headheight
\advance \topmargin by -\headsep
\evensidemargin \oddsidemargin
  \def\C{{\mathbb{C}}}           \def\N{{\mathbb{N}}}   \def\Q{{\mathbb{Q}}}         \def\Z{{\mathbb{Z}}}

\def\bfa{{\bf{a}}} \def\bfb{{\bf{b}}} \def\bfc{{\bf{c}}}      \def\bfi{{\bf{i}}}                 

 \def\cB{{\mathcal{B}}}           \def\cM{{\mathcal{M}}}   \def\cP{{\mathcal{P}}} \def\cQ{{\mathcal{Q}}} \def\cR{{\mathcal{R}}}   \def\cU{{\mathcal{U}}}     

               \def\sP{{\mathscr{P}}}          

\newcommand{\Ga}{\Gamma}
\newcommand{\ga}{\gamma}

\newcommand{\eps}{\epsilon}

\newcommand\Ca{{\bf \operatorname{{\bf Cantor}}}}

\newcommand\Closed{\operatorname{Closed}}

\newcommand\Factor{{\operatorname{Factor}}}

\newcommand\IRS{\operatorname{IRS}}

\newcommand\Prob{\operatorname{Prob}}

\newcommand\Stab{\operatorname{Stab}}
\newcommand\Sub{\operatorname{Sub}}

\def\cc{{\curvearrowright}}

\def\mc{\mathcal}
\newcommand{\Es}[1]{\EuScript{#1}}

\usepackage{upgreek}
\newcommand{\Ra}{\Rightarrow}
\newcommand{\ra}{\rightarrow}
\newcommand{\csuchthat}{\, :\,}

\begin{document}
\title{The space of stable weak equivalence classes of measure-preserving actions}
\author{Lewis Bowen\footnote{supported in part by NSF grant DMS-1500389, NSF CAREER Award DMS-0954606} ~and Robin Tucker-Drob\footnote{supported in part by NSF grant DMS-1600904}}
\maketitle

\begin{abstract}
The concept of (stable) weak containment for measure-preserving actions of a countable group $\Ga$ is analogous to the classical notion of (stable) weak containment of unitary representations. If $\Ga$ is amenable then the Rokhlin lemma shows that all essentially free actions are weakly equivalent. However if $\Ga$ is non-amenable then there can be many different weak and stable weak equivalence classes. Our main result is that the set of stable weak equivalence classes naturally admits the structure of a Choquet simplex. For example, when $\Ga=\Z$ this simplex has only a countable set of extreme points but when $\Ga$ is a nonamenable free group, this simplex is the Poulsen simplex. We also show that when $\Ga$ contains a nonabelian free group, this simplex has uncountably many strongly ergodic essentially free extreme points.
\end{abstract}

\noindent
{\bf Keywords}: weak containment, pmp actions\\
{\bf MSC}:37A35\\

\noindent
\tableofcontents

\section{Introduction}

A. Kechris introduced the notion of weak containment for group actions as an analogue of weak containment for unitary representations \cite[II.10 (C)]{Kechris-global-aspects}. Given a countable group $\Ga$ and probability measure-preserving (pmp) actions $\bfa:=\Ga \cc^a (X,\mu), \bfb:=\Ga \cc^b (Y,\nu)$ on standard probability spaces, we say $\bfa$ is {\bf weakly contained} in $\bfb$ (denoted $\bfa \prec \bfb$) if for every finite measurable partition $\{P_i\}_{i=1}^n$ of $X$, finite $S \subseteq \Ga$ and $\epsilon>0$ there exists a measurable partition $\{Q_i\}_{i=1}^n$ of $Y$ satisfying
$$| \mu(\ga^a P_i \cap P_j) - \nu(\ga^b Q_i \cap Q_j) | < \eps$$
for all $\ga \in S$ and $1\le i,j \le n$ (where the action of $\Ga\cc^a X$ is denoted $\ga^a x$ for $\ga \in \Ga, x\in X$ for example).  We say $\bfa$ is {\bf weakly equivalent} to $\bfb$, denoted $\bfa \sim \bfb$, if both $\bfa \prec \bfb$ and $\bfb \prec \bfa$.

The Rokhlin Lemma is essentially equivalent to the statement that for the group $\Ga=\Z$ all essentially free\footnote{An action is essentially free if almost every point has trivial stabilizer.} pmp actions are weakly equivalent. Indeed, as remarked in \cite{kechris-2012}, this statement holds for all countable amenable groups. However it fails for nonamenable groups because strong ergodicity is an invariant of weak equivalence \cite[Prop. 10.6]{Kechris-global-aspects}. This motivates the problem of providing a description of the set of all weak equivalence classes, denoted by $\Es{W}_\Ga$,  for a given group $\Ga$. 

We start with an equivalent definition of weak containment. Let $\Ca$ denote any space homeomorphic to a Cantor set. Let $\Ga$ act on $\Ca^\Ga$ by $(\ga x)(f)=x(\ga^{-1} f)$. Let $\Prob_\Ga(\Ca^\Ga)$ denote the space of all $\Ga$-invariant Borel probability measures on $\Ca^\Ga$ equipped with the weak* topology. It is well-known that $\Prob_\Ga(\Ca^\Ga)$ is a {\bf Choquet simplex}: this means it is a compact convex subset of a locally convex topological vector space with the property that every element $\mu \in \Prob_\Ga(\Ca^\Ga)$ can be uniquely written as a convex integral of extreme points of $\Prob_\Ga(\Ca^\Ga)$.

Given an action $\bfa:=\Ga \cc^a (X,\mu)$, let $\Factor(\bfa) \subseteq \Prob_\Ga(\Ca^\Ga)$ denote the set of measures of the form $\Phi_*\mu$ where $\Phi:X \to \Ca^\Ga$ is a $\Ga$-equivariant measurable map and $\Phi_*\mu = \mu \circ \Phi^{-1}$. The weak* closure of $\Factor(\bfa)$ is denoted $W(\bfa)$. It follows from  \cite{abert-weiss-2013} that $\bfa \prec \bfb$ if and only if $W(\bfa) \subseteq W(\bfb)$ (see also \cite[Prop. 3.6]{T-D12}). So the map $\bfa \mapsto W(\bfa)$ induces an injective map from the set of weak equivalence classes into the set of closed subsets of $\Prob_\Ga(\Ca^\Ga)$. We equip the latter with the Vietoris topology, and $\Es{W}_\Ga$ with the subspace topology. This topology, considered in \cite{T-D12}, is a reformulation of a construction due to Abert-Elek. The main result of \cite{AE11} is that $\Es{W}_\Ga$ is compact (an alternative proof is given in \cite{T-D12}).

This motivates the question: what sort of subsets of $\Prob_\Ga(\Ca^\Ga)$ can have the form $W(\bfa)$? This is addressed in \cite{abert-weiss-2013}: if $\bfa$ is strongly ergodic then $W(\bfa)$ is contained in the set of extreme points of $\Prob_\Ga(\Ca^\Ga)$. If $\bfa$ is ergodic but not strongly ergodic then $W(\bfa)$ is a subsimplex of $\Prob_\Ga(\Ca^\Ga)$: that is, it is the convex hull of the extreme points of $\Prob_\Ga(\Ca^\Ga)$ contained in $W(\bfa)$. See Theorem \ref{thm:stable} below.

We now turn towards a description of stable weak equivalence classes where we obtain a more complete picture. We say that $\bfa$ is {\bf stably weakly contained} in $\bfb$, denoted $\bfa \prec_s \bfb$, if $\bfa \times \bfi \prec \bfb \times \bfi$ where $\bfi$ denotes the trivial action of $\Ga$ on the unit interval equipped with Lebesgue measure. If both $\bfa \prec_s \bfb$ and $\bfb \prec_s \bfa$ then we say the two actions are {\bf stably weakly equivalent} and denote this by $\bfa \sim_s \bfb$.  Let $SW(\bfa):=W(\bfa \times \bfi)$; by \cite[Theorem 1.1]{T-D12} $SW(\bfa )$ is the closed convex hull of $W(\bfa )$ (see Lemma \ref{lem:stable1}). Then $\bfa \prec_s \bfb$ if and only of $SW(\bfa) \subseteq SW(\bfb)$. So $\bfa \mapsto SW(\bfa)$ induces an injective map from the set of stable weak equivalence classes into the set of closed convex subsets of $\Prob_\Ga(\Ca^\Ga)$. We denote the set of stable weak equivalence classes with the induced topology by $\Es{SW}_\Ga$. Like the weak equivalence case, $\Es{SW}_\Ga$ is compact\footnote{This can be proven in a manner similar to the case of $\Es{W}_\Ga$. Alternatively, by Lemma \ref{lem:stable1} one can view $\Es{SW}_\Ga$ as the subspace of convex elements of $\Es{W}_\Ga$. Because convexity is a closed property, $\Es{SW}_\Ga$ is closed in $\Es{W}_\Ga$ and therefore is compact.}. By Theorem \ref{thm:stable}, if $\bfa$ is ergodic then $SW(\bfa)$ is a subsimplex of $\Prob_\Ga(\Ca^\Ga)$.

To simplify notation, let $\sP:=\Prob_\Ga(\Ca^\Ga)$ and $\Closed(\sP)$ denote the space of all closed subsets of $\sP$ equipped with the Vietoris topology, and let $\mathrm{CloCon}(\sP)$ denote the collection of all closed convex subsets of $\sP$. The space $\mathrm{CloCon}(\sP)$ is compact, and it admits a natural convex structure: if $F_1,F_2 \in \mathrm{CloCon}(\sP)$ and $t\in [0,1]$ then $tF_1+(1-t)F_2 \in \mathrm{CloCon}(\sP)$ is defined to be the set of all measures of the form $t\mu_1 + (1-t)\mu_2$ with $\mu_i \in F_i$ ($i=1,2$). The space $\Es{SW}_{\Ga}$ is then a closed convex subset of $\mathrm{CloCon}(\sP )$. Our main result is that $\Es{SW}_\Ga$ is a Choquet simplex (Theorem \ref{thm:simplex}). This means that for every $\alpha \in \Es{SW}_\Ga$ there exists a {\em unique} probability measure on the set of extreme points of $\Es{SW}_\Ga$ such that $\alpha$ is the barycenter of this measure.

Can we identify the simplex $\Es{SW}_\Ga$ up to affine homeomorphism? To begin answering this question we need the following concept. An {\bf invariant random subgroup} is a random subgroup of $\Ga$ whose law is invariant under conjugation. Let $\IRS(\Ga)$ denote the space of all conjugation-invariant Borel probability measures on the space of subgroups of $\Ga$. To any pmp action $\bfa=\Ga \cc^a (X,\mu)$ we associate the element $\IRS(\bfa)$ defined by
$$\IRS(\bfa)  : = \Stab_*\mu$$
where $\Stab:X \to \Sub(\Ga)$ is the map $\Stab(x)=\{g\in \Ga:~g^ax=x\}$ and $\Sub(\Ga)$ is the space of subgroups of $\Ga$ with the pointwise convergence topology. By \cite{AE11} and \cite{T-D12}, if $\bfa \sim_s \bfb$ then $\IRS(\bfa) = \IRS(\bfb)$. So we have a well-defined map $\IRS: \Es{SW}_\Ga \to \IRS(\Ga)$. In \cite[Theorem 5.2]{T-D12} and \cite[Corollary 5.1]{burton-weak-2015} it is shown that this map is affine and continuous. It is also surjective by \cite[Proposition 45]{abert-IRS-2014}.  In \cite{T-D12} (see the remark after \cite[Theorem 1.8]{T-D12}), it is shown that when $\Ga$ is amenable, $\IRS$ is a homeomorphism. So we have a complete description of $\Es{SW}_\Ga$ in the case where $\Ga$ is amenable.

When $\Ga$ is nonamenable however, there can be many stable weak equivalence classes which map to a given IRS of $\Ga$. If $\Ga$ is a nonamenable free group, then P. Burton showed that the subsimplex of $\Es{SW}_\Ga$ consisting of all stable weak equivalence class of free actions, is a Poulsen simplex \cite{burton-weak-2015}. This means that its extreme points are dense. There is a unique Poulsen simplex up to affine homeomorphism \cite{lindenstrauss1978poulsen}. If $\Ga$ has property (T), then Theorem \ref{thm:T} below shows that $\Es{SW}_\Ga$ is a Bauer simplex which means that the extreme points form a closed subset of $\Es{SW}_\Ga$. In particular, $\Es{SW}_\Ga$ cannot be a Poulsen simplex.

In case $\Ga$ has a nonamenable free subgroup, Theorem \ref{thm:uncountable} below shows that $\Es{SW}_\Ga$ has an uncountable set $\{S_p\}_{p\ge 2}$ of extreme points indexed by the interval $[2,\infty)$. Moreover, each $S_p$ is the class of a free, mixing, strongly ergodic action. The proof uses Okayasu's result that the universal $\ell^p(\Ga)$-representations of the free group are pairwise weakly inequivalent \cite{Oka14}.

\subsection{Related literature}

Burton and Kechris have written a very recent survey article on weak containment \cite{burton-kechris}.

For every countable group $\Ga$ there exists a pmp action $\bfa$ such that {\em all} pmp actions of $\Ga$ are weakly contained in $\bfa$. This is known as the {\bf weak Rokhlin property} \cite{MR2285247}. This property was introduced by Glasner-King where it was shown to imply a correspondence between generic properties of pmp actions and invariant measures \cite{GK98}.

Moreover, every essentially free action weakly contains every Bernoulli action \cite{abert-weiss-2013}. This latter fact has been used to show that the cost of essentially free actions of $\Ga$ is maximized by the Bernoulli actions. Moreover, certain combinatorial quantities such as independence number of actions are weak equivalence invariants which allows one to use compactness to prove that their extreme values are realized \cite{CKT13}. This paper also establishes equivalent definitions of weak containment in terms of the space of all actions and ultraproducts of actions.

A residually finite group $\Ga$ has {\bf property MD} if every action is stably weakly contained in a profinite action of $\Ga$. It is known that residually finite amenable groups, free groups, and fundamental groups of closed hyperbolic 3-manifolds\footnote{In \cite{bowen-tucker-2013} it was shown that fundamental groups of virtually fibered hyperbolic 3-manifolds have property MD. By \cite{agol-virtual-haken} all closed hyperbolic 3-manifolds are virtually fibered.} have property MD \cite{bowen-tucker-2013}. This property is a strengthening of Lubotsky-Shalom's {\bf property FD} which is defined similarly but for unitary representations instead of pmp actions \cite{lubotzky-shalom-2004}. It is unknown whether the direct product of two free groups has MD or FD.

The main result of \cite{AE12} is that, for strongly ergodic actions, weak containment of a given finite action implies actual containment of the same action. They apply this to show that certain groups such as free groups and linear property (T) groups, admit an uncountable family of non-weakly-equivalent essentially free ergodic actions \cite{AE12}. Ioana and Tucker-Drob strengthened the main result of \cite{AE12} by generalizing finite actions to distal actions.  Consequently, the weak equivalence class of a strongly ergodic action remembers the weak isomorphism class of its maximal distal factor \cite{MR3545932}.

Aaserud and Popa introduced several variants of weak containment in the context of orbit-equivalence  \cite{aaserud-popa}. Ab\'{e}rt and Elek show in \cite{AE11} that the invariant random subgroup (IRS) of an action is a weak equivalence invariant. Tucker-Drob showed in \cite{T-D12} that actions within a given weak equivalence class are unclassifiable up to countable structures.

Peter Burton showed in \cite{burton-weak-2015} that the space of stable weak equivalence classes naturally forms a convex compact subset of a Banach space and, when $\Ga$ is amenable, identifies this simplex as the simplex of IRS's. The proofs used some ideas from an earlier draft of this paper.

{\bf Acknowledgements}. After obtaining the proof that the space of stable weak equivalence classes forms a simplex, we naturally wondered what simplex could it be. It seemed natural to guess that for the free group, one obtains a Poulsen simplex. Peter Burton's beautiful proof of this result inspired us to finish this work \cite{burton-weak-2015}. So thanks, Peter. We would also like to thank Matthew Wiersma for pointing us to Okayasu's paper \cite{Oka14}.

\section{Preliminaries}

\subsection{Glossary}

\begin{itemize}
\item An action $\Ga \cc (X,\mu)$ is {\bf pmp} if $\mu$ is a probability measure and the action is measure-preserving.
\item An action $\Ga \cc (X,\mu)$ is {\bf essentially free} if for a.e. $x\in X$, the stabilizer of $x$ in $\Ga$ is trivial.
\end{itemize}

\subsection{Notation}

Throughout this paper, $\Ca$ denotes the Cantor set, $\Ga$ a countable group, $\sP:=\Prob_\Ga(\Ca^\Ga)$ the space of invariant Borel probability measures on $\Ca^\Ga$ equipped with the weak* topology, $\sP^{erg}\subseteq \sP$ the subspace of ergodic invariant measures, $\Closed(\sP)$ the space of closed subsets of $\sP$ with the Vietoris topology, and $\mathrm{CloCon}(\sP)$ the space of closed convex subsets of $\sP$. Moreover, if $\bfa = \Ga \cc ^a (X,\mu)$ is a pmp action then $\Factor(\bfa) \subseteq \sP$ is the set of all measures of the form $\Phi_*\mu$ where $\Phi:X \to \Ca^\Ga$ is measurable and $\Ga$-equivariant. Also $W(\bfa)$ is the weak* closure of $\Factor(\bfa)$ and $SW(\bfa) = W(\bfa \times \bfi)$ where $\bfi$ denotes the trivial action of $\Ga$ on the unit interval with respect to Lebesgue measure. We let $\Es{W}_\Ga \subseteq \Closed(\sP)$ denote the collection of all closed subsets of the form $W(\bfa)$ and $\Es{SW}_\Ga \subseteq \Closed(\sP)$ denotes the collection of all closed subsets of the form $SW(\bfa)$ over all pmp actions $\bfa$ of $\Ga$. Note that $\Es{SW}_{\Ga}\subseteq \mathrm{CloCon}(\sP )$ by \cite[Theorem 1.1]{T-D12}.

If $\bfa = \Ga \cc ^a (X,\mu)$ then the action of $\Ga$ on $X$ is denoted $g^ax$ for $g\in \Ga, x\in X$. For $t>0$ we define the action $t \bfa$ by $t\bfa = \Ga \cc ^a (X,t\mu)$. In other words, it is the same action, we simply scale the measure by $t$. If $\bfb=\Ga \cc^b (Y,\nu)$ is another action then we define $\bfa \oplus \bfb$ to be the action $\bfa \oplus \bfb = \Ga \cc ^{a \oplus b} (X \sqcup Y,\mu \oplus \nu)$ where $X \sqcup Y$ denotes the disjoint union of $X$ and $Y$, $\mu \oplus \nu(E) = \mu(E \cap X) + \nu(E \cap Y)$ for $E \subseteq X \sqcup Y$ and $g^{a\oplus b}x = g^ax, g^{a\oplus b}y = g^by$ for $x \in X$, $y\in Y$ and $g\in \Ga$.

\section{Strong ergodicity}

\begin{defn}
Let $\bfa = \Ga \cc ^a (X,\mu)$. We say that a sequence $\{B_i\}_{i=1}^\infty$ of measurable sets in $X$ is {\bf asymptotically invariant} (with respect to $\bfa$) if for every $g \in \Ga$,
$$\lim_{i\to\infty} \mu(B_i \vartriangle g^a B_i) = 0.$$
We say that $\{B_i\}_{i=1}^\infty$ is {\bf nontrivial} if $\limsup_{i\to\infty} \mu(B_i)(1-\mu(B_i)) >0$. The action $\bfa$ is {\bf strongly ergodic} if it does not admit any nontrivial asymptotically invariant sequences. Equivalently, $\bfa$ is strongly ergodic if $\bfb\prec\bfa$ implies $\bfb$ is ergodic (see \cite[Prop. 5.6]{CKT13}).
\end{defn}

\begin{defn}
If $\bfa$ and $\bfb$ are pmp actions of $\Ga$ and $t\in [0,1]$ then we write $t\bfb\prec \bfa$ to mean that $t\bfb\oplus (1-t)\bfi_0\prec \bfa$ where $\bfi_0$ is the trivial action of $\Ga$ on a one point probability space. Since any pmp action trivially contains $\bfi_0$, if $\bfc$ is any pmp action and $s\bfb\oplus (1-s)\bfc \prec \bfa$ for some $0<s\leq 1$, then $s\bfb \prec \bfa$.

More generally, if $\bfa, \bfb$ are any finite-measure-preserving actions then $\bfb \prec \bfa$ means that $t\bfb \prec t \bfa$ where $t>0$ is chosen so that $t\bfa$ is probability-measure-preserving.
\end{defn}

The main result of this section is:
\begin{thm}\label{thm:not-strongly-ergodic}
Let $\bfa$ be an ergodic but not strongly ergodic pmp action of $\Ga$. Then for every $0<t<1$, $t\bfa \prec \bfa$.
\end{thm}

The next result was obtained in \cite[Proof of Lemma 2.3]{jones-schmidt-1987}.
\begin{lem}[Asymptotically invariant sets are mixing]\label{lem:not-strongly-ergodic}
Let $\bfa = \Ga \cc ^a (X,\mu)$ be ergodic and let $\{B_i\}_{i=1}^\infty \subseteq X$ be an asymptotically invariant sequence with respect to $\bfa$ such that
$$\lim_{i\to\infty} \mu(B_i) = t ~\textrm{ for some } ~0<t<1.$$
If $A_1,A_2$ are any measurable subsets of $X$ then for every $g\in \Ga$,
$$\lim_{i\to\infty} | \mu(B_i \cap A_1 \cap g^a A_2) - \mu(B_i) \mu(A_1 \cap g^a A_2)| = 0.$$

\end{lem}

\begin{cor}
If $\bfa=\Ga \cc^a (X,\mu)$ is an ergodic but not strongly ergodic pmp action of $\Ga$ then for every $t \in (0,1)$ there exists an asymptotically invariant sequence $\{B_i\}$ such that $\lim_{i\to\infty} \mu(B_i) = t$.
\end{cor}

\begin{proof}
Let $N \subseteq (0,1)$ be the set of all numbers $t \in (0,1)$ such that there exists an asymptotically invariant sequence $\{B_i\}$ such that $\lim_{i\to\infty} \mu(B_i) = t$. Suppose that $\{B_i\}$ and $\{C_j\}$ are asymptotically invariant sequences. Then $\{X \setminus B_i\}, \{B_i \cap C_i\}$ and $\{B_i \cup C_i\}$ are asymptotically invariant. From the previous lemma it follows that  $\{1-t,st,s+t-st:~s,t\in N\} \subseteq N$. Since $N$ is closed and nonempty, it follows that $N = (0,1)$ as claimed.
\end{proof}

\begin{proof}[Proof of Theorem \ref{thm:not-strongly-ergodic}]
Let $\bfa=\Ga \cc^a (X,\mu)$, $\cP=\{P_1,\ldots, P_k\}$ be a finite Borel partition of $X$ and $0<t<1$. By the previous corollary there exists an asymptotically invariant sequence $\{B_n\}$ with $\lim_{n\to\infty} \mu(B_n) = t$. By Lemma \ref{lem:not-strongly-ergodic},
$$\lim_{n\to\infty} |  \mu(B_n \cap P_i\cap g^a P_j) - t\mu(P_i \cap g^a P_j)| = 0$$
for all $P_i,P_j\in \cP$ and $g\in \Ga$. Set $Q^{(n)}_i = B_n \cap P_i$. The asymptotic invariance of $\{B_n\}$ and the previous limit implies
$$\lim_{n\to\infty} |  \mu(Q^{(n)}_i \cap g^a Q^{(n)}_j) - t\mu(P_i \cap g^a P_j)| = 0$$
for any $i,j$ and $g\in \Ga$. This implies the theorem.
\end{proof}

\section{Ergodic decomposition}

The main purpose of this section is to prove:
\begin{thm}\label{thm:weak-decomposition}
Let $\bfa=\Ga \cc ^a (X,\mu )$, $\bfb = \Ga \cc ^b (Y,\nu )$, and $\bfc= \Ga\cc ^c (Y',\nu ' )$ be pmp actions of $\Ga$. Let us assume $\bfa$ is ergodic.
\begin{enumerate}
\item If $\bfa\prec s\bfb\oplus (1-s)\bfc$ for some $0< s\leq 1$ then $\bfa \prec \bfb$. Moreover $\bfa$ is weakly contained in almost every ergodic component of $\bfb$.
\item If $s\bfb\oplus (1-s)\bfc \prec \bfa$ for some $0<s\leq 1$ then $\bfb\prec \bfa$. Moreover, almost every ergodic component of $\bfb$ is weakly contained in $\bfa$.
\item If $s\bfb\oplus (1-s)\bfc \sim \bfa$ for some $0< s\leq 1$ then $\bfb\sim \bfa$. Moreover almost every ergodic component of $\bfb$ is weakly equivalent to $\bfa$.
\end{enumerate}
\end{thm}

Part (1) is equivalent to \cite[Theorem 3.12]{T-D12}. Part (3)  follows from parts (1) and (2). So we need only prove part (2). We will need measure algebras as defined next.

\begin{defn}[Measure algebras]
Let $(X,\mu)$ denote a measure-space. Given measurable sets $A,B \subseteq X$ we say that $A$ and $B$ are $\mu$-equivalent if $\mu(A \vartriangle B)=0$. Let $A^{\mu}$ denote the $\mu$-equivalence class of $A$. The {\bf measure-algebra} of $\mu$, denoted $\mbox{MALG}_\mu$, is the set of all classes $A^{\mu}$ where $A \subseteq X$ is a measurable set of finite measure. We usually abuse notation by treating an element of $\mbox{MALG}_\mu$ as if it were a subset of $X$ instead of an equivalence class.

The set $\mbox{MALG}_\mu$ has a natural metric given by symmetric difference: the distance between $A, B\in \mbox{MALG}_\mu$ is $\mu(A\vartriangle B)$. Note that if $\mu$ is a standard $\sigma$-finite measure then $\mbox{MALG}_\mu$ is separable; it contains a countable dense subset.
\end{defn}


We need the next few lemmas before proving the second statement of Theorem \ref{thm:weak-decomposition}.

\begin{lem}\label{lem:JS}
Let $\bfa=\Ga \cc ^a (X,\mu )$ and $\bfb= \Ga \cc ^b (Y,\nu )$ be pmp actions of $\Ga$ and let $t\in [0,1]$. Suppose that $\bfa$ is ergodic and that $t\bfb\prec \bfa$. Then given any Borel partition $B_0,\dots , B_{m-1}$ of $Y$, finite subset $F\subseteq \Ga$, $\epsilon >0$ and finite subset $\mc{A}\subseteq \mbox{MALG}_\mu$, there exist $B_0',\dots ,B_{m-1}'\subseteq X$ such that, letting $B' = \bigcup _{j<m}B_j'$, we have the following for all $g\in F$:
\begin{enumerate}
\item $\mu (B') = t$ and $\mu (g ^a B' \vartriangle B')<\epsilon$;
\item $\sum_{i,j<m} |\mu (g ^a B_i'\cap B_j' ) - t\nu (g ^b B_i\cap B_j )| <\epsilon$;
\item $\sum_{A_1,A_2 \in \mc{A} } |\mu (B' \cap A_1 \cap  g^a A_2) - t\mu (A_1 \cap  g^a A_2)|<\epsilon$.
\end{enumerate}
\end{lem}

\begin{proof}
Let $B_0,\dots , B_{m-1}$, $F$, $\epsilon$, and $\mc{A}$ be given as in the statement of the Lemma. Fix an increasing exhaustive sequence $F_0\subseteq F_1\subseteq \cdots$ of finite subsets of $\Ga$, along with a sequence of real numbers $\epsilon _n >0$ with $\epsilon _n \ra 0$. Since $t\bfb\prec \bfa$, for each $n\in \N$ we may find subsets $B^{(n)}_0,\dots , B^{(n)}_{m-1}\subseteq X$ such that, letting $B^{(n)} = \bigcup _{j<m}B_j^{(n)}$,  we have for all $g\in F_n$,
\begin{enumerate}
\item[(i)] $|\mu (B^{(n)}) - t|<\epsilon_n$ and $\mu ( g ^a B^{(n)} \vartriangle B^{(n)})<\epsilon_n$;
\item[(ii)] $\sum_{i,j<m} |\mu ( g ^a B_i^{(n)}\cap B_j^{(n)} ) - t\nu ( g ^b B_i\cap B_j )| <\epsilon_n$.
\end{enumerate}
Because $\bfa$ is ergodic, either $\mu$ has no atoms or it uniformly distributed on a finite set of atoms. In the first case we can add or subtract a small subset from some of the $B_i^{(n)}$'s to ensure the equality $\mu(B^{(n)})=t$ (at the cost of replacing the error tolerance $\epsilon_n$ with $C\epsilon_n$ for some fixed constant $C$). In the second case we will automatically have this equality once $\epsilon_n$ is sufficiently small. In either case, we may assume $\mu(B^{(n)})=t$.

Now (i) says that the sequence $\{ B^{(n)} \}$ is an asymptotically invariant sequence for $\bfa$. So Lemma \ref{lem:not-strongly-ergodic} now implies that there exists an $n$ such that $B'_j:=B^{(n)}_j$ satisfies this lemma.
%
%
%
%
\end{proof}

\begin{lem}\label{lem:tb}
Let $\bfa=\Ga \cc ^a (X,\mu )$ and $\bfb= \Ga \cc ^b (Y,\nu )$ be pmp actions of $\Ga$. Let $t\in [0,1]$. Suppose that $\bfa$ is ergodic and $t\bfb \prec \bfa$. Then $t\bfb\oplus (1-t)\bfa\prec \bfa$.
\end{lem}

\begin{proof}
Given Borel partitions $\{B_0 ,\dots , B_{m-1}\}$ of $Y$ and $\{A_0,\dots ,A_{n-1}\}$ of $X$ along with a finite subset $F\subseteq \Ga$ and $\epsilon >0$ it suffices to find Borel subsets $B_0',\dots ,B_{m-1}', A_0',\dots ,A_{n-1}'\subseteq X$ such that
\begin{enumerate}
\item[(i)] $A_i'\cap B_j' = \emptyset$ for all $i<n$ and $j<m$;
\item[(ii)] $\sum_{i,j<n} |\mu ( g ^aA_i' \cap A_j' )- (1-t)\mu ( g ^a A_i \cap A_j )|<2\epsilon$ for all $g \in F$;
\item[(iii)] $\sum_{i,j<m} |\mu ( g ^aB_i' \cap B_j' )- t\nu ( g ^b B_i \cap B_j )|<\epsilon$ for all $g \in F$.
\end{enumerate}
Let $\mc{A} = \{ A_i\}_{i=0}^{n-1} \subseteq \mbox{MALG}_\mu$. 
By hypothesis we have $t\bfb\prec \bfa$, so we may find sets $B_0',\dots ,B_{m-1}'\subseteq X$ and $B'=\bigcup _{j<m} B_j'$, satisfying (1), (2), and (3) of Lemma \ref{lem:JS}. Let $A_i' = A_i\setminus B'$. Then (i) and (iii) are clearly satisfied and it remains to show (ii). Given $ g \in F$ and $i,j<m$ we have $ g ^a A_i' \cap A_j' = ( g^a A_i \cap A_j) \cap ( g^a (X\setminus B')\cap (X\setminus B'))$. So
\begin{align*}
&~~~~|\mu ( g ^aA_i' \cap A_j' )- (1-t)\mu ( g ^a A_i \cap A_j )| \\
&\leq |\mu (( g ^a A_i \cap A_j )\cap (X\setminus B')) - (1-t)\mu ( g^a A_i \cap A_j)| + \mu ( g^a (X\setminus B')\vartriangle (X\setminus B')) \\
&< 2\epsilon,
\end{align*}
where the first term is at most $\epsilon$ by property (3) from Lemma \ref{lem:JS}, and the second term is at most $\epsilon$ by property (1) from that lemma. Therefore,
$$\sum_{i,j<n} |\mu ( g ^aA_i' \cap A_j' )- (1-t)\mu ( g ^a A_i \cap A_j )|<2\epsilon n^2.$$
Since $n$ is fixed we can replace $\epsilon$ with $\epsilon/n^2$ to satisfy (ii).
\end{proof}


\begin{proof}[Proof of Theorem \ref{thm:weak-decomposition} part (2)]
Assume that $s\bfb\oplus (1-s)\bfc \prec \bfa$ for some $0<s\leq 1$. This immediately implies  $s\bfb \prec \bfa$. Let $r_n = \sum _{k=0}^n s(1-s)^k$. We show by induction on $n\geq 0$ that $r_n\bfb \prec \bfa$. We have $r_0\bfb = s\bfb\prec \bfa$ by hypothesis. Assume for induction that $r_n\bfb\prec \bfa$. Then
\[
r_{n+1}\bfb = (s+ (1-s)r_n)\bfb \prec s\bfb\oplus (1-s)(r_n\bfb) \prec s\bfb\oplus (1-s)\bfa \prec \bfa
\]
where the last weak containment follows from Lemma \ref{lem:tb}. Since $r_n\bfb\prec \bfa$ for all $n$ and $\lim _n r_n = s\sum _{k=0}^\infty (1-s)^k = 1$ it follows that $\bfb\prec \bfa$.

Next we assume that $\bfb \prec \bfa$. Let $\nu = \int _{z\in Z} \nu _z \, d\eta$ be the disintegration of $\nu$ corresponding to the ergodic decomposition of $\bfb$, and for each $z\in Z$ let $\bfb_z = \Ga \cc ^b (Y,\nu _z )$. We must show that $\bfb_z\prec \bfa$ almost surely. Let $C =\{ z\in Z\, :\, \bfb_z \not\prec \bfa \}$. Suppose toward a contradiction that $\eta (C) >0$. Let $\Es{B}$ be a countable Boolean algebra which generates the Borel sigma algebra on $Y$.  For each finite subset $\cQ \subseteq \Es{B}$, consider the space $[0,1]^{\cQ \times \cQ}$ of all functions $\delta :  \cQ \times \cQ \to [0,1]$. This space is separable so there exists a countable dense subset $\Delta_{\cQ} \subseteq [0,1]^{\cQ \times \cQ}$.

Let $\Es{I}$ denote the set of all quadruples $(F, \mc{Q}, \delta ,\epsilon )$ where $F \subseteq \Ga$ is finite, $\mc{Q}\subseteq \cB$ is finite, $\delta:F \times \cQ \times \cQ \to [0,1]$ is such that $\delta(g,\cdot,\cdot)  \in \Delta_{\cQ}$ for all $g\in F$ and  $\epsilon \in (0,1)\cap \Q$.  Let $\Es{I}_0\subseteq \Es{I}$ denote the subset consisting of all $(F, \mc{Q}, \delta ,\epsilon )\in \Es{I}$ for which there does not exist any function $f:\mc{Q}\ra \mbox{MALG}_\mu$ satisfying
$$\sum_{B,B' \in \cQ} |\mu (g^a f(B) \cap f(B') ) - \delta (g ,B,B')|\le 3\epsilon$$
for all $g\in F$. For each $(F, \mc{Q}, \delta ,\epsilon ) \in \Es{I}_0$ define the set
\[
C_{F,\mc{Q},\delta ,\epsilon} := \left\{  z\in C\, : \, \forall g \in F,~\sum_{B,B' \in \cQ} |\nu_z (g^{b_z} B \cap B' ) - \delta (g,B,B')|\le \epsilon   \right\} .
\]
It follows from the definitions that $C= \bigcup _{(F,\mc{Q},\delta ,\epsilon )\in \Es{I}_0} C_{F,\mc{Q},\delta ,\epsilon }$. Since this is a countable union and $\eta (C)>0$ we must have $\eta (C_{F_0,\mc{Q}_0,\delta _0,\epsilon _0 }) =t >0$ for some quadruple $(F_0,\mc{Q}_0,\delta _0,\epsilon _0 )\in \Es{I}_0$. Let $C_0 = C_{F_0,\mc{Q}_0,\delta _0,\epsilon _0}$ and define
\begin{align*}
\bfb_0 &= \int _{z\in C_0} \bfb_z \, d\eta _{C_0} = \Ga \cc ^b (Y, \nu _{C_0}) \\
\bfb_1 &= \int _{z\in Z\setminus C_0} \bfb_z \, d\eta _{Z\setminus C_0} = \Ga \cc ^b (Y,\nu _{Z\setminus C_0})
\end{align*}
where $\eta _{C_0}$ is the normalized restriction of $\eta$ to $C_0$ and $\nu _{C_0} = \int _z\nu _z \, d\eta _{C_0}$, and similarly for $\eta _{Z\setminus C_0}$ and $\nu _{Z\setminus C_0}$. Then $\bfa\succ \bfb \cong t \bfb_0 \oplus (1-t)\bfb_1$. So by the first part of this proof we have $\bfa\succ \bfb_0$.


Since $\bfb_0\prec \bfa$ there exists some $f:\mc{Q}_0\ra \mbox{MALG}_\mu$ such that
$$\sum_{B,B' \in \cQ} |\mu (g ^a f(B) \cap f(B') ) - \delta (g ,B,B')|\le 2\epsilon_0$$
for all $g \in F_0$ which contradicts that $(F_0,\mc{Q}_0,\delta _0,\epsilon _0 )\in \Es{I}_0$.
\end{proof}

\section{Stable weak equivalence classes}

The purpose of this section is to prove:
\begin{thm}\label{thm:stable}
If $\bfa$ is ergodic then $SW(\bfa)$ is a subsimplex of $\Prob_\Ga(\Ca^\Ga)$. In other words, it is a closed convex subset whose extreme points are extreme points of $\Prob_\Ga(\Ca^\Ga)$. Moreover, $\bfa$ is strongly ergodic if and only if $W(\bfa)$ is the set of extreme points of $SW(\bfa)$, in which case $SW(\bfa)$ is a Bauer simplex. If $\bfa$ is ergodic but not strongly ergodic then $SW(\bfa)=W(\bfa)$ is a Poulsen simplex.
\end{thm}



\begin{lem}\label{lem:stable1}
For any pmp action $\bfa$ of $\Ga$, $SW(\bfa)$ is the closed convex hull of $W(\bfa)$.
\end{lem}

\begin{proof}
We may assume $\bfa=\Ga \cc^a (X,\mu_a)$. We first show that $SW(\bfa)$ contains the closed convex hull of $W(\bfa)$. So let $t_1,\ldots, t_n>0$ with $\sum_i t_i=1$ and $\mu_1,\ldots, \mu_n \in W(\bfa)$. It suffices to show that $\sum_i t_i\mu_i \in SW(\bfa)$. By definition there exist factor maps $\varphi_{ij}:X \to \Ca^\Ga$ such that $\lim_{j\to\infty} \varphi_{ij*}\mu_a = \mu_i$ for all $i$. Define $\Phi_j:X \times [0,1] \to \Ca^\Ga$ by $\Phi_j(x,t) = \varphi_{ij}(x)$ if $i$ is such that $\sum_{k<i} t_k < t \le \sum_{k \le i} t_k$. It follows that $\Phi_{j*}(\mu_a \times \textrm{Leb})= \sum_i t_i \varphi_{ij*}\mu_a$. So $\lim_{j\to\infty} \Phi_{j*}(\mu_a \times \textrm{Leb})= \sum_i t_i\mu_i$ as required.

Next we show $SW(\bfa)$ is contained in the closed convex hull of $W(\bfa)$. So let $\Phi:X\times [0,1] \to \Ca^\Ga$ be a factor map. Let $\phi_t$ be the restriction of $\Phi$ to $X \times \{t\}$. Observe that $\phi_t$ is also a factor map and
$$\Phi_{*}(\mu_a \times \textrm{Leb}) = \int_0^1 \phi_{t*}\mu_a~dt.$$
Because $\phi_t$ can be regarded as factor map of $\bfa$, this shows that $\Phi_{*}(\mu_a \times \textrm{Leb})$ is contained in the closed convex hull of $W(\bfa)$. Because $\Phi$ is arbitrary, $SW(\bfa)$ is contained in the closed convex hull of $W(\bfa)$.
\end{proof}



\begin{lem}\label{lem:thing1}
Let $\bfa$ be an ergodic but not strongly ergodic action pmp action of $\Ga$. Then $W(\bfa)=SW(\bfa)$.
\end{lem}

\begin{proof}
By Theorem \ref{thm:not-strongly-ergodic} and Lemma \ref{lem:tb}, $t\bfa \oplus (1-t)\bfa \prec \bfa$ for any $t\in (0,1)$. By induction, this implies $\oplus_i t_i \bfa \prec \bfa$ for any sequence $t_1,\ldots, t_n >0$ with $\sum_i t_i=1$.  In other words, $W(\oplus_i t_i \bfa) \subseteq W(\bfa)$. However, $W(\oplus_i t_i \bfa)$ contains $\oplus_i t_i W(\bfa)$ where the latter is defined to be the collection of all measures of the form $\sum_i t_i \mu_i$ with $\mu_i \in W(\bfa)$. Thus $W(\bfa)$ is convex. Lemma \ref{lem:stable1} now implies $W(\bfa)=SW(\bfa)$.

\end{proof}

\begin{proof}[Proof of Theorem \ref{thm:stable}]
To prove the first statement, suppose $\nu \in SW(\bfa) \subseteq \Prob_\Ga(\Ca^\Ga)$  is not ergodic. So we can write it as $\nu = t\nu_1 + (1-t)\nu_2$ for some $\nu_1,\nu_2 \in \Prob_\Ga(\Ca^\Ga)$ such that $\nu_1$ and $\nu_2$ are mutually singular and $t \in (0,1)$. However, Part 2 of Theorem \ref{thm:weak-decomposition} implies that $\nu_1,\nu_2 \in SW(\bfa)$. Therefore, $\nu$ cannot be an extreme point of $SW(\bfa)$. This proves that all extreme points of $SW(\bfa)$ are extreme points of $\Prob_\Ga(\Ca^\Ga)$.

If $\bfa$ is strongly ergodic then it follows immediately that every measure in $W(\bfa)$ is ergodic and therefore extreme in $\Prob_\Ga(\Ca^\Ga)$. Since $SW(\bfa)$ is the closed convex hull of $W(\bfa)$ this handles this case.

Now suppose $\bfa$ is ergodic but not strongly ergodic. To see that the extreme points are dense, observe that every measure in $\Factor(\bfa)$ is ergodic (hence extreme) and $SW(\bfa)=W(\bfa)$ is the weak* closure of $\Factor(\bfa)$ by Lemma \ref{lem:thing1}.

\end{proof}

\section{Compactness}
For simplicity, in this section we let $\sP=\Prob_\Ga(\Ca^\Ga)$. This is a compact metrizable space in the weak* topology. Let $\rm{Closed}(\sP)$ be the space of all closed subsets of $\sP$ with the Vietoris topology with respect to which $\rm{Closed}(\sP)$ is a compact metrizable space. Let $\Es{W}_\Ga:=\{W(\bfa)\}_{\bfa} \subseteq \rm{Closed}(\sP)$ and $\Es{SW}_\Ga := \{ SW(\bfa)\}_{\bfa}\subseteq \rm{Closed}(\sP)$.  In \cite{T-D12} it is proven that the topologies induced on $\Es{W}_\Ga$ and $\Es{SW}_\Ga$ from their inclusions into $\Closed(\sP)$ are equivalent to the topologies defined in \cite{AE11} (another proof is in \cite[Theorem 3.1]{burton-weak-2015}). The next theorem is the main result of \cite{AE11}:

\begin{thm}\label{thm:compact}
Both $\Es{W}_\Ga$ and $\Es{SW}_\Ga$ are closed subsets of $\rm{Closed}(\sP)$. Therefore, $\Es{W}_\Ga$ and $\Es{SW}_\Ga$ are compact metrizable spaces.
\end{thm}

For each $\rho \in \sP$ let $W(\rho ):=W(\bfa) \subseteq \sP$ where $\bfa=\Ga \cc (\Ca^\Ga,\rho)$. Similarly, let $SW(\rho):=SW(\bfa)$. We will frequently make use of the following facts:
\begin{enumerate}
\item[(1)] For every pmp action $\bfa =\Ga \cc ^a (X,\mu )$ there is a measure $\eta \in \sP$ such that $\Ga \cc (\Ca ^\Ga , \eta )$ is isomorphic to $\bfa$.
\item[(2)] For any two pmp actions $\bfa _0 =\Ga \cc ^{a_0} (X_0,\mu _0 )$ and $\bfa _1 =\Ga \cc ^{a_1} (X_1,\mu _1 )$ of $\Gamma$, there are measures $\eta _0, \eta _1 \in \sP$ whose supports are disjoint such that $\Ga \cc (\Ca ^\Ga , \eta _0)$ is isomorphic to $\bfa _0$ and $\Ga \cc (\Ca ^\Ga , \eta _1)$ is isomorphic to $\bfa _1$
\end{enumerate}

Clearly (1) follows from (2). To see (2), let $C_0$ and $C_1$ be nonempty disjoint clopen subsets of $\Ca$ and for $i=0,1$, let $\varphi _i :X_i\rightarrow C_i$ be injections, and define $\Phi _i :X_i \rightarrow \Ca ^{\Ga}$ by $\Phi _i (x)(g) = \varphi _i ((g^{-1})^{a_i}x)$. Then $\Phi _i$ is injective and equivariant, and the supports of $\eta _i := (\Phi _i)_*\mu _i$ are contained in $C_i^{\Ga}$, so the measures $\eta _0$, $\eta _1$ work.

\medskip

We introduce some notation which will be useful throughout the rest of the paper.

\begin{notation}
To ease notation, we will not distinguish between a measure $\mu \in \sP$ and the corresponding action $\Ga \cc (\Ca^\Ga,\mu)$. For example, we will say that a measure $\mu \in \sP$ is ergodic or essentially free if the corresponding action is. Similarly if $\rho_1,\rho_2 \in \sP$ we will write $\rho_1 \prec \rho_2$ to mean that the action corresponding to $\rho_1$ is weakly contained in the action corresponding to $\rho_2$.
\end{notation}

\subsection{Lower semi-continuity}
As a corollary to Theorem \ref{thm:compact}, we will show that $SW$ is lower semi-continuous as a map from $\sP$ to $\Es{SW}_\Ga$. In general, if $C_1,C_2,\ldots \subseteq \sP$ are closed subsets then we define $\liminf_i C_i$ to be the set of all $\mu_\infty \in \sP$ such that there exist $\mu_i \in C_i$ (for $i\in \N$) such that $\lim_i \mu_i = \mu_\infty$.

\begin{cor}\label{cor:semicontinuity}[$SW$ is lower semi-continuous]
If $\{\mu_i\}_i$ is a sequence in $\sP$ and $\lim_i \mu_i = \mu_\infty$ then
$$SW(\mu_\infty) \subseteq \liminf_i SW(\mu_i).$$
\end{cor}

\begin{remark}
$SW$ is not continuous in general. For example, consider the case when $\Ga=\Z$. It is possible to find a sequence of measures $\mu_i \in \sP$ such that $\Ga \cc (\Ca^\Ga,\mu_i)$ is essentially free for all $i$ but $\lim_i \mu_i = \delta_x$ is the Dirac measure on a fixed point $x \in \Ca^\Ga$. By the Rokhlin Lemma, $SW(\mu_i)=\sP$ for all $i$ and $SW(\mu_i) \ne  SW(\delta_x)$ since $SW(\delta_x)$ is the subspace of measures supported on fixed points.
\end{remark}

\begin{proof}
Since $\Es{SW}_\Ga$ is compact, after passing to a subsequence, we may assume that $\lim_i SW(\mu_i) = SW(\nu)$ for some $\nu \in \sP$. Since $\mu_\infty = \lim_i \mu_i$ it follows that $\mu_\infty \in SW(\nu)$. Thus $\mu_\infty \prec_s \nu$ and therefore $SW(\mu_\infty) \subseteq SW(\nu)$.
\end{proof}

\section{Convex integrals and couplings}\label{sec:properties}
Let $\sP^{erg}$ denote the extreme points of $\sP=\Prob_\Ga(\Ca^\Ga)$. Let $\Prob(\sP^{erg})$ denote the space of Borel probability measures on $\sP^{erg}$. Let $\uppi : \Ca^\Ga\ra \sP^{erg}$ be an ergodic decomposition map. By definition this means that $\uppi$ is a $\Ga$-invariant Borel map satisfying
\begin{itemize}
\item For each $e\in \sP^{erg}$, $e(\{ x \in \Ca^\Ga \csuchthat \uppi (x) =e \} ) =1$.
\item For each $\mu \in \sP$, $\mu = \int _{e\in \sP^{erg}} e \, d\uppi _* (\mu )$.
\end{itemize}
Furthermore, $\uppi$ is unique in the following sense: if $\uppi '$ is another such map then the set $\{ x \csuchthat \uppi (x) \neq \uppi ' (x) \}$ is $\mu$-null for all $\mu \in \sP$ \cite{MR1784210}.

Let $\uppi  _* : \sP \ra \Prob(\sP^{erg})$ be the associated affine map which takes a measure $\mu \in \sP$ to its ergodic decomposition $\uppi _* (\mu )\in \Prob(\sP^{erg})$. In what follows we will abuse notation and write $\uppi (\mu )$ for $\uppi _*(\mu )$. If $\kappa \in \Prob(\sP)$ then we let $\upbeta (\kappa )  \in \sP$ denote the Barycenter of $\kappa$. By definition,
$$\upbeta(\kappa) = \int _{\sP} \mu ~d\kappa( \mu ).$$
So $\upbeta (\uppi (\mu )) = \mu$, and if $\kappa \in \Prob(\sP^{erg})$ then $\uppi(\upbeta (\kappa )) =\kappa$.

\begin{defn}
Let $(X,\mc{A},\mu )$ and $(Y,\mc{B},\nu )$ be probability spaces. A {\bf coupling} of $\mu$ with $\nu$ is a probability measure $\rho$ on $(X\times Y,\mc{A}\otimes\mc{B})$ such that $(\mbox{proj}_X)_*\rho = \mu$ and $(\mbox{proj}_Y)_*\rho = \nu$. 

Let $(Z,\mc{C},\eta )$ be another probability space and let $\rho$ be a coupling of $\mu$ with $\nu$, and let $\sigma$ be a coupling of $\nu$ with $\eta$. Then the {\bf composition} of $\rho$ and $\sigma$, denoted $\rho\circ \sigma$, is the coupling of $\mu$ with $\eta$ defined by $\rho\circ \sigma = \int _Y \rho _y \times \sigma ^y \, d\nu$, where $\rho = \int _Y \rho _y \times \updelta _y \, d\nu$ and $\sigma = \int _Y \updelta _y \times \sigma ^y \, d\nu$ are the respective disintegrations of $\rho$ and $\sigma$ via the natural projection maps.
\end{defn}

%
\begin{lem}\label{lem:semiaffine}
Let $\lambda$ and $\omega$ be Borel probability measures on $\sP$ and assume that there is a coupling $\rho$ of $\lambda$ with $\omega$ which concentrates on the set $\{ (\mu ,\nu ) \, : \, \mu \prec _s \nu \}$. Assume in addition that there is a $\omega$-conull set $\sP _{\omega} \subseteq \sP$ such that the measures in $\sP_{\omega}$ are mutually singular. Then $\upbeta (\lambda ) \prec _s \upbeta (\omega )$.
\end{lem}

We note that the hypothesis on $\omega$ is automatically satisfied if $\omega$ concentrates on $\sP ^{erg}$.

\begin{proof}
Let $\rho = \int _{\sP_{\omega}} \rho ^{\nu } \times \updelta _{\nu} \, d\omega (\nu )$ be the disintegration of $\rho$ over $\omega$. Then for $\omega$-a.e.\ $\nu$, the measure $\rho ^{\nu }$ concentrates on $SW(\nu )$, hence $\upbeta (\rho ^{\nu}) \in SW(\nu )$, since $SW(\nu )$ is a closed convex set. We have $\upbeta (\lambda ) = \int \upbeta (\rho ^{\nu}) \, d\omega (\nu )$. Fix an atomless Borel probability measure $\nu_0$ on $\Ca$. Also, let $\Ga$ act on $\Ca^\Ga \times \Ca$ by
$$g(x,y) = (gx,y).$$
Then $\upbeta (\rho ^{\nu}) \prec \nu \times \nu_0$ for $\omega$-a.e.\ $\nu$. Fix a Borel partition $\cP=\{P_1,\ldots, P_k\}$ of $\Ca^\Ga$, $\epsilon>0$ and a finite subset $F \subseteq \Ga$. It suffices to show there exists a Borel partition $\{U_1,\ldots, U_k\}$ of $\Ca^\Ga \times \Ca$ such that
$$\left| \int _{\sP} \upbeta (\rho ^{\nu})(P_i \cap gP_j) - \nu \times \nu_0(U_i \cap gU_j) ~d\omega( \nu ) \right| < \epsilon$$
for every $g\in F$ and $1\le i,j \le k$.

Let $\{\cQ^{(n)}\}_{n=1}^\infty$ be an enumeration of all clopen partitions of $\Ca^\Ga \times \Ca$ of the form $\cQ^{(n)}=\{Q^{(n)}_1,\ldots, Q^{(n)}_k\}$. There are only countably many such partitions. For $\omega$-a.e.\ $\nu$, since $\upbeta (\rho ^{\nu}) \prec _s \nu \times \nu_0$, and because the clopen sets are dense in the measure algebra of $\nu \times \nu_0$, there exists some number $n(\nu ) \in \N$ such that
$$\left|  \upbeta (\rho ^{\nu }) (P_i \cap gP_j) - \nu \times \nu_0\left(Q^{(n(\nu ))}_i \cap gQ^{(n(\nu ))}_j\right) \right| < \epsilon.$$
for every $g\in F$ and $1\le i,j \le k$. We choose $n(\nu )$ to be the smallest natural number with this property. With this choice, the map $\nu \mapsto n(\nu )$ is measurable.

Let $\cM$ denote the set of all $m\in \N$ such that
$$\omega(\{ \nu \in \sP :~ n(\nu )=m\}) >0.$$
Define
$$\kappa_m = \int_{\nu :~n(\nu )=m} \nu \times \nu_0~d\omega( \nu ).$$
This is a $\Ga$-invariant Borel measure on $\Ca^\Ga\times \Ca$. Moreover, the measures $\{\kappa_m:~m \in \cM\}$ are mutually singular since the measures in the $\omega$-conull set $\sP _{\omega}$ are mutually singular. So there exists a Borel partition $\cR=\{R_m\}_{m \in \cM}$ of $\Ca^\Ga \times \Ca$ such that
$$\kappa_m(\Ca^\Ga \times \Ca) = \kappa_m(R_m)$$
and $R_m$ is $\Ga$-invariant for all $m\in \cM$. Thus for $\omega$-a.e. $\nu\in \sP$ we have
$$\nu \times \nu_0(E) = \nu \times \nu_0(E \cap R_{n(\nu )})$$
for any Borel $E \subseteq \Ca^\Ga \times \Ca$.  Let
$$U_i = \bigcup_{m\in \cM} R_m \cap Q_i^{(m)}.$$
Then $\{U_1,\ldots, U_k\}$ is a Borel partition of $\Ca^\Ga \times \Ca$ and for any $g\in F$, $1\le i, j \le k$,
\begin{eqnarray*}
&&\left| \int \upbeta (\rho ^{\nu} ) (P_i \cap gP_j) - \nu \times \nu_0(U_i \cap gU_j) ~d\omega( \nu ) \right|\\
 &\le& \int |\upbeta (\rho ^{\nu}) (P_i \cap gP_j) - \nu \times \nu_0(U_i \cap gU_j)| ~d\omega( \nu )\\
&=& \sum_{m\in \cM} \int_{\nu :~n(\nu )=m} |\upbeta (\rho ^{\nu}) (P_i \cap gP_j) - \nu \times \nu_0(U_i \cap gU_j)| ~d\omega( \nu )\\
&=& \sum_{m\in \cM} \int_{\nu :~n(\nu )=m} \left|\upbeta (\rho ^{\nu})(P_i \cap gP_j) - \nu \times \nu_0\left(Q^{(m)}_i \cap gQ^{(m)}_j\right)\right| ~d\omega( \nu )\\
&\le& \epsilon.
\end{eqnarray*}

\end{proof}

\section{Coupling Theorem}
The main theorem of this section is:

\begin{thm}\label{thm:Scoupling}[Coupling Theorem]
Let $\mu , \nu \in \sP$.
\begin{enumerate}
\item[(i)] $\mu \prec _s\nu$ if and only if there exists a coupling $\rho$ of $\uppi (\mu )$ and $\uppi (\nu )$ which concentrates on the set $\{ (e_0,e_1) \in \sP^{erg}\times \sP^{erg} \csuchthat e_0\prec _s e_1 \}$

\item[(ii)] $\mu \sim _s \nu$ if and only if there exists a coupling $\rho$ of $\uppi (\mu )$ and $\uppi (\nu )$ which concentrates on the set $\{ (e_0,e_1) \in \sP^{erg}\times \sP^{erg}\csuchthat e_0\sim _s e_1 \}$
\end{enumerate}
Moreover, if $\mu \sim _s \nu$ and $\rho$ is any coupling of $\uppi (\mu )$ and $\uppi (\nu )$ which concentrates on $\{ (e_0,e_1) \in \sP^{erg}\times \sP^{erg}\csuchthat e_0\prec _s e_1 \}$, then $\rho$ in fact concentrates on $\{ (e_0, e_1) \in \sP^{erg}\times \sP^{erg}\csuchthat e_0 \sim _s e_1 \}$.
\end{thm}

Before proving this, we need to investigate properties of a natural basis for the topology of $\Es{SW}_\Ga$.

\begin{defn}
To each open subset $U$ of $\sP$ we associate the sets
\begin{align*}
B_U &= \{ \rho \in \sP \csuchthat SW(\rho )\cap U\neq \emptyset \} \\
C_U &= \{ \rho \in \sP \csuchthat \uppi  (\rho ) (B_U) > 0 \} .
\end{align*}
\end{defn}

The following proposition gives some basic properties of the sets $B_U$ and $C_U$ which will be used several times below.

\begin{prop}\label{prop:facts}Let $U$ and $V$ be open subsets of $\sP$.
\begin{itemize}
\item[(i)] $C_U\cap {\sP^{erg}} = B_U \cap \sP^{erg}$.

\item[(ii)] $U\subseteq B_U$ and $U\cap {\sP^{erg}} \subseteq C_U$.

\item[(iii)] If $\mu \in B_U$ and $\mu \prec_s \nu$ then $\nu \in B_U$.

\item[(iv)] If $U\subseteq V$ then $B_U\subseteq B_V$ and $C_U\subseteq C_V$.

\item[(v)] $B_U$ is open and $C_U$ is Borel.
%
\end{itemize}
\end{prop}

\begin{proof}
Statements (i) through (iv) all follow from the definitions. For (v), to see $B_U$ is open it suffices to show that $\sP\setminus B_U$ is closed. Assume $\rho _n \in \sP\setminus B_U$ and $\rho _n \ra \rho \in \sP$. Then $SW(\rho _n ) \subseteq \sP\setminus U$ for all $n$, so $\liminf _n SW(\rho _n ) \subseteq \sP\setminus U$ since $\sP\setminus U$ is closed. By Lemma \ref{cor:semicontinuity}, $SW(\rho ) \subseteq \liminf _n SW(\rho _n )\subseteq \sP\setminus U$, i.e., $\rho \in \sP\setminus B_U$. The set $C_U$ is Borel since $\uppi$ and $B_U$ are both Borel.
%
\end{proof}

\begin{lem}\label{lem:singleErg} Let $V\subseteq \sP$ be open.
\begin{enumerate}
\item[(1)] Let $\mu \in C_V$. Then for any $e\in \sP^{erg}\setminus B_V$ there exists a neighborhood $U$ of $\mu$ with $e\not\in B_U$.

\item[(2)] Let $L\subseteq C_V$ be compact, and let $\nu \in \sP$. Then for any $\epsilon >0$ there exists an open set $U\subseteq \sP$ with $L\subseteq U$ and $\uppi (\nu )(B_U\setminus B_V)< \epsilon$.

\item[(3)] Let $\lambda$ be a Borel probability measure on $\sP$, and let $\nu\in \sP$. Then for any $\epsilon >0$ there exists an open set $U\subseteq \sP$ with $\lambda (C_V\setminus U) = 0$ and $\uppi (\nu )(B_U\setminus B_V)<\epsilon$.
\end{enumerate}
\end{lem}

\begin{proof}
(1): Assume toward a contradiction that there is some $e\in \sP^{erg}\setminus B_V$ such that for all open neighborhoods $U$ of $\mu$ we have $e\in B_U$, i.e., $SW(e)\cap U\neq \emptyset$. This means that $\mu \in SW(e)$, so that $\mu \prec _s e$ and therefore
\begin{equation}\label{eqn:combine}
\uppi  (\mu ) (\{ e' \in \sP^{erg}\csuchthat e'\prec _s e \} )=1
\end{equation}
by Theorem \ref{thm:weak-decomposition} (2). From (\ref{eqn:combine}) and the hypothesis $\mu \in C_V$ we conclude that there is some $e' \in \sP^{erg}\cap B_V$ with $e' \prec _s e$. Therefore, by Proposition \ref{prop:facts}, $e\in B_V$, a contradiction.

(2): Fix $\epsilon >0$. Let $\{ O_n \} _{n\in \N}$ be a countable basis of open subsets of $\sP$ and let $\{ U_n \} _{n\in \N}$ enumerate all finite unions of elements of $\{ O_n \} _{n\in \N}$.

\begin{claim}
Let $e \in \sP^{erg}\setminus B_V$. Then there exists some $n\in \N$ such that $L\subseteq U_n$ and $e\not\in B_{U_n}$.
\end{claim}

\begin{proof}[Proof of Claim]
By part (1), for each $\mu \in L$ there is some $n(\mu )\in \N$ such that $\mu \in O_{n(\mu )}$ and $e\not\in B_{O_{n(\mu )}}$. Then $L \subseteq \bigcup _{\mu \in L} O_{n(\mu )}$, and since $L$ is compact there exists some finite $Q\subseteq L$ such that $L\subseteq \bigcup _{\mu \in Q}O_{n(\mu )}$. Taking any $n\in \N$ with $U_n = \bigcup _{\mu \in Q}O_{n(\mu )}$ works since  $e\not\in \bigcup _{\mu \in Q}B_{O_{n(\mu )}} = B_{U_n}$. \qedhere[Claim]
\end{proof}

For each $e\in \sP^{erg}\setminus B_V$ let $n(e) = \min \{ n \in \N \csuchthat L\subseteq U_{n} \mbox{ and } e\not\in B_{U_n} \}$. Let $N$ be so large that
\[
\uppi (\nu ) (\{ e \in \sP^{erg}\setminus B_V \csuchthat n(e) < N \} ) > \uppi (\nu ) (\sP^{erg}\setminus B_V) -\epsilon ,
\]
and define $U= \bigcap \{ U_{n} \csuchthat n<N \mbox{ and }L \subseteq U_n \}$. Then $U$ is open and $L\subseteq U$. Furthermore, $\sP^{erg}\cap B_U\setminus B_V \subseteq \{ e \in \sP^{erg}\setminus B_V \csuchthat n(e) \geq N \}$ since if $e\not\in B_V$ is such that $n(e)<N$, then $U\subseteq U_{n(e)}$ and therefore $e\not\in B_U$ (since $e\not\in B_{U_{n(e)}}$). This shows that $\uppi (\nu )(B_U\setminus B_V) <\epsilon$.

(3): The measure $\lambda$ is regular, so we may find a sequence $L_1, L_2,\dots$, of compact subsets of $C_V$ with $\lambda (C_V\setminus L_n )\rightarrow 0$. For each $n$ apply (2) to find an open $U_n$ with $L_n\subseteq U_n$ and $\uppi (\nu )(B_{U_n}\setminus B_V)< \epsilon /2^n$. Let $U=\bigcup _n U_n$. Then $\lambda (C_V\setminus U) = 0$, and $B_U = \bigcup _n B_{U_n}$, hence $\uppi (\nu )(B_U\setminus B_V)< \epsilon$.
\end{proof}

\subsection{Ultrapowers of measure spaces}
Let $\cU$ denote a nonprincipal ultrafilter on $\N$ and $(X,\mu)$ be a standard Borel probability space. Define an equivalence relation $\sim_\cU$ on $X^\N$ by $\{x_i\} \sim_\cU \{y_i\}$ if and only if $\{ n\in \N \csuchthat x_n=y_n \} \in \mc{U}$. Let $X_\cU:=X^\N/\sim_\cU$ denote the set of all $\sim_\cU$ equivalence classes. If $\{B_n\}$ is a sequence of subsets of $X$ then we let $[B_n] \subseteq X_\cU$ denote the set of all equivalence classes of the form $[x_n]$ with $\{n \in \N:~x_n \in B_n\}\in \mc{U}$. For each Borel $B\subseteq X$ we also let $[B]\subseteq X_{\mc{U}}$ denote the set $[B]:= \{ [x_n] \csuchthat \{ n \csuchthat x_n \in B \} \in \mc{U} \}$ corresponding to the constant sequence.

If $B_n \subseteq X$ is a sequence of Borel sets then we define $\mu_\cU([B_n]) := \lim_{n\to\cU} \mu(B_n)$. This function extends in a unique way to a probability measure, still denoted $\mu _{\cU}$, on the sigma-algebra $\mc{B}(X_{\cU})$ generated by all sets of the form $[B_n]$ where each $B_n \subseteq X$ is Borel. We let $\sigma (\mu _{\cU})$ denote the completion of $\mc{B}(X_{\cU})$ with respect to $\mu _{\cU}$. Thus $(X_\cU , \mu_\cU )$ (equipped with the sigma algebra $\sigma (\mu _{\cU} )$) is a probability space called the {\bf ultrapower} of $(X,\mu)$. In general, it is not standard because the corresponding measure algebra need not be separable. See \cite{CKT13} for more details on $(X_\cU , \mu_\cU )$.

There is a natural measure algebra embedding $I: \mathrm{MALG}_{\mu} \hookrightarrow \mathrm{MALG}_{\mu_{\cU}}$ given by $B^\mu\mapsto [B]^{\mu _{\cU}}$. The map $I$ preserves the algebra structure and it is continuous, hence it also preserves the $\sigma$-algebra structure. If we assume that $X$ is a compact Polish space, then the following proposition shows that the limit map $[x_n] \mapsto \lim _{n\rightarrow \cU}x_n$, gives a natural point realization of the embedding $I$.

\begin{prop}\label{prop:limU}
\begin{itemize}
\item[(1)] Let $K$ be a compact Polish space. Let $\varphi _n : X\rightarrow K$, $n\in \N$, be a sequence of Borel functions from $X$ to $K$. Then the function $\varphi : X_{\mc{U}}\rightarrow K$ given by $\varphi ([x_n]) = \lim _{n\rightarrow \mc{U}}\varphi _n (x_n)$ is measurable.

\item[(2)] Assume that $X$ is a compact Polish space. Then the map $\mathrm{lim} _{\cU} : X_{\cU}\rightarrow X$, defined by $\mathrm{lim} _{\cU}([x_n]) = \lim _{n\rightarrow \cU}x_n$, is measurable, and for each Borel $B\subseteq X$ we have $\mu _{\cU} ( \mathrm{lim} _{\cU}^{-1}(B) \triangle [B] ) = 0$. In particular, $\mathrm{lim}_{\cU} : (X_{\cU}, \mu _{\cU}) \rightarrow (X,\mu )$ is measure preserving.
\end{itemize}
\end{prop}

\begin{proof}
For (1), let $d$ be a compatible metric on $K$ and fix an open set $V\subseteq K$. Since $V$ is open we have $V = \bigcup _m V_m$ where $V_m = \{ k\in V \csuchthat d(k , K\setminus V ) > 1/m \}$. We then have the equality $\varphi ^{-1}(V) = [\varphi _n^{-1}(V_1)]\cup [\varphi _n ^{-1}(V_2)]\cup\cdots$, which shows $\varphi$ is measurable. 
Statement (2) corresponds to the case $X=K$, $\varphi _n = \mathrm{id}_X$ for all $n$, and $\varphi = \mathrm{lim} _{\mc{U}}$. In this case, using the notation from above, the sequence $[V_1], [V_2],\dots$ increases to $\mathrm{lim}_{\cU}^{-1}(V)$, and thus $I(V_m^{\mu}) \rightarrow \mathrm{lim} _{\cU} ^{-1}(V)^{\mu _{\cU}}$ in $\mathrm{MALG}_{\mu _{\cU}}$. But also $I(V_m ^{\mu })\rightarrow I(V^{\mu})$ by continuity of $I$, hence $I(V^{\mu})= \mathrm{lim} _{\cU} ^{-1}(V)^{\mu _{\cU}}$. Thus, the collection $\mathcal{B}$, of all Borel subsets $B\subseteq X$ satisfying $\lim _{\cU}^{-1}(B)^{\mu _{\cU}} = I(B^{\mu})$, contains all open subsets of $X$, and it is also a $\sigma$-algebra since the maps $B\mapsto \lim _{\cU}^{-1}(B)^{\mu _{\cU}}$ and $B\mapsto I(B^{\mu})$ both preserve $\sigma$-algebra operations. This shows $\mathcal{B}$ contains every Borel set, and completes the proof of (2).
\end{proof}

\subsection{Proof of the Coupling Theorem}

\begin{proof}[Proof of Theorem \ref{thm:Scoupling}]
(i): Assume first that there exists a coupling $\rho$ of $\uppi (\mu )$ and $\uppi (\nu )$ as in (i). Then the disintegration of $\rho$ with respect to the right projection map $(e_0,e_1)\mapsto e_1$ is of the form $\rho = \int \rho ^e \times \updelta _e \, d\uppi (\nu )(e)$. For $\uppi (\nu )$-almost every $e\in \sP^{erg}$ the measure $\rho ^e$ concentrates on $\{ e' \csuchthat e'\prec _s e \}$. Since $SW(e)$ is convex (by Theorem \ref{thm:stable}),  $\upbeta (\rho ^e) \prec _s e$. By Lemma \ref{lem:semiaffine},  $\mu = \upbeta (\uppi (\mu ) ) = \int  \upbeta (\rho ^e ) \, d\uppi (\nu )(e) \prec _s \int  e \, d\uppi (\nu )(e)= \nu$.

Now assume that $\mu \prec _s \nu$. Let $\nu ' \in \sP$ be such that $\Ga \cc ( \Ca ^{\Ga}, \nu ' )$ is isomorphic to the product of $\Ga \cc ( \Ca ^{\Ga}, \nu )$ with the identity action of $\Ga$ on $([0,1],\mathrm{Leb})$. Then there is a coupling $\sigma$ of $\uppi (\nu ')$ and $\uppi (\nu )$ which concentrates on pairs of isomorphic ergodic components. If we can find a coupling $\rho '$ of $\uppi (\mu )$ and $\uppi (\nu ')$ which concentrates on pairs $(e_0,e_1)$ with $e_0\prec e_1$, then the composition $\rho = \rho '\circ \sigma$ will be the desired coupling of $\uppi(\mu)$ with $\uppi(\nu)$. Therefore, after replacing $\nu$ by $\nu '$ if necessary, we may assume without loss of generality that $SW(\nu ) = W(\nu )$ so that in fact $\mu \prec \nu$.

Fix a non-principal ultrafilter $\mc{U}$ on $\N$. Let $(\Ca^\Ga_\cU, \nu _{\mc{U}})$ denote the ultrapower of $(\Ca^\Ga,\nu)$.  As $\uppi (\nu )$ is a measure on $\sP$ which concentrates on $\sP^{erg}$, the ultrapower $\uppi (\nu )_{\mc{U}}$ is a measure on the space ${\sP}_{\mc{U}}$ which concentrates on the set $[\sP^{erg} ]$ which we identify with $\sP^{erg}_{\mc{U}}$. By Proposition \ref{prop:limU}, we have that $\uppi (\nu ) = \mathrm{lim}_{\mc{U}}{}_*\uppi (\nu )_{\mc{U}}$. In particular, for $\uppi (\nu )_{\mc{U}}$-almost every $[e_n]\in \sP^{erg}_{\mc{U}}$, we have $\lim _{n\rightarrow \mc{U}}e_n \in \sP^{erg}$. For each $[e_n]\in \sP^{erg}_{\mc{U}}$ we let $\prod _{\mc{U}}[e_n]$ denote the measure on $\Ca^\Ga_\cU$ determined by $\prod _{\mc{U}}[e_n]([A_n])=\lim _{n\ra \mc{U}}e_n(A_n)$ for $A_n\subseteq \Ca^\Ga$ Borel. 

\indent Since $\mu \prec \nu$ there exist Borel factor maps $\Phi _n :\Ca^\Ga \ra \Ca^\Ga$ with $(\Phi _n)_*\nu \ra \mu$. Let $\Phi : (\Ca^\Ga )_{\mc{U}} \ra \Ca^\Ga$ be the ultralimit function given by $\Phi ([x_n])=\lim _{n\ra \mc{U}} \Phi _n (x_n)$. By \cite[Proposition 3.11]{T-D12} we have $\Phi _*(\nu _{\mc{U}}) = \lim _{n\ra \mc{U}}(\Phi _n)_*\nu = \mu$ and $\Phi _* \prod _{\mc{U}}[e_n] = \lim _{n\ra \mc{U}}(\Phi _n)_*e_n$ for every $[e_n] \in \sP^{erg}_{\mc{U}}$. The map $[e_n]\mapsto \Phi _* \prod_{\mc{U}}[e_n]= \lim _{n\ra \mc{U}}(\Phi _n)_*e_n$, from $\sP ^{erg}_{\mc{U}}$ to $\sP$, is therefore measurable by Proposition \ref{prop:limU}. By \cite[Proposition A.1]{T-D12}, the decomposition $\nu = \int  e \, d\uppi (\nu )(e)$ yields $\nu _{\mc{U}} = \int \textstyle{\prod _{\mc{U}}}[e_n] \, d\uppi (\nu )_{\mc{U}} ({[e_n]} )$, and hence
\begin{equation}\label{eqn:disint}
\mu = \Phi _* (\nu _{\mc{U}}) = \int \lim _{n\ra \mc{U}}(\Phi _n)_*e_n \, d\uppi (\nu )_{\mc{U}}([e_n]) .
\end{equation}
Let $\rho$ be the measure on $\sP \times \sP$ defined by
\begin{align}
\rho &= \label{eqn:rhoU} \int  \uppi \left( \lim _{n\ra \mc{U}}(\Phi _n)_*e_n \right)\times \updelta _{\lim _{n\rightarrow\mc{U}}e_n} \ d\uppi (\nu )_{\mc{U}}([e_n]).
\end{align}
Then $\rho$ concentrates on $\sP ^{erg} \times \sP ^{erg}$, and \eqref{eqn:disint} and Proposition \ref{prop:limU} show that $\rho$ is a coupling of $\uppi (\mu )$ and $\uppi (\nu )$.

\begin{claim} \label{claim:keyclaim} Let $V\subseteq \sP$ be open. Then $\rho (B_V\times (\sP^{erg}\setminus B_V)) = 0$.
\end{claim}

\begin{proof}[Proof of Claim]
Suppose not. Then the expression (\ref{eqn:rhoU}) implies that $\uppi (\nu )_{\mc{U}}(D_0)>0$, where
\[
D_0 = \left\{ [e_n] \in \sP^{erg}_{\mc{U}}\csuchthat \lim _{n\ra \mc{U}} (\Phi _n)_*e_n\in C_V \mbox{ and }\lim _{n\ra\mc{U}}e_n \not\in B_V \right\} .
\]
Let $\lambda$ denote the push-forward of $\uppi (\nu )_{\mc{U}}$ under the map $[e_n]\mapsto \lim _{n\ra \mc{U}}(\Phi _n)_*e_n$, so that $\lambda$ is a Borel probability measure on $\sP$. By Lemma \ref{lem:singleErg}.(3) we may find an open set $U\subseteq \sP$ such that $\lambda (C_V\setminus U ) = 0$ and $\uppi (\nu )(B_U\setminus B_V) < \uppi (\nu )_{\mc{U}} (D_0)$. Thus, for $\uppi (\nu )_{\mc{U}}$-almost every $[e_n] \in D_0$ we have $\lim _{n\ra \mc{U}} (\Phi _n)_*e_n\in U$ and (by Proposition \ref{prop:limU}) $[e_n]\not\in [B_V]$. Therefore, $\uppi (\nu )_{\mc{U}}(D_0)\leq \uppi (\nu ) _{\mc{U}}(D_1)$, where
\[
D_1 = \left\{ [e_n] \in \sP^{erg}_{\mc{U}}\csuchthat \lim _{n\ra \mc{U}} (\Phi _n)_*e_n\in U \mbox{ and }[e_n]\not\in [B_V] \right\} .
\]
Since $\uppi (\nu )(B_U\setminus B_V) < \uppi (\nu )_{\mc{U}}(D_0) \leq \uppi (\nu )_{\mc{U}}(D_1)$, the set $D_1 \setminus ([B_U]\setminus [B_V])$ is $\nu _{\mc{U}}$-non-null and hence nonempty. Fix any $[e_n]$ in this set. Then $[e_n]\not\in [B_V]$ (since $[e_n]\in D_1$), and $[e_n]\not\in [B_U]\setminus [B_V]$, hence
\begin{equation}\label{eqn:contrad}
[e_n]\not\in [B_U] .
\end{equation}
On the other hand, $[e_n]\in D_1$ implies $\lim _{n\ra \mc{U}} (\Phi _n)_*e_n \in U$. Since $U$ is an open neighborhood about $\lim _{n \ra \mc{U}}(\Phi _n )_*e_n$ we have $\{ n \csuchthat (\Phi _n )_*e_n \in U \} \in \mc{U}$. For each $n$ with $(\Phi _n )_*e_n \in U$ we have $(\Phi _n )_*e_n \in SW(e_n)\cap U$ and so $e_n \in B_U$. Therefore $\{ n \csuchthat e_n \in B_U \} \in \mc{U}$, i.e., $[e_n] \in [B_U]$, which contradicts \eqref{eqn:contrad}. \qedhere[Claim \ref{claim:keyclaim}]
\end{proof}

Let $\{ V_i \} _{i\in \N }$ be a countable base of open subsets of $\sP$. Then
\[
\{ (e_0,e_1)\csuchthat e_0\prec_s e_1 \} = \bigcap _{i\in \N } \{ (e_0, e_1) \csuchthat e_0 \in B_{V_i}\Ra e_1\in B_{V_i} \} ,
\]
and $\rho$ concentrates on this set by Claim \ref{claim:keyclaim}.

(ii): If $\rho$ is a coupling of $\uppi (\mu )$ and $\uppi (\nu )$ as in (ii), then in particular $\rho (\{ (e_0,e_1)\csuchthat e_0\prec _s e_1 \} ) =1$, so $\mu \prec _s \nu$ by part (i). Similarly, $\nu \prec _s \mu$, and thus $\nu \sim _s \mu$. The other direction of (ii) will follow from (i) once we establish the final statement of the theorem.

Suppose $\mu \sim _s \nu$ and let $\rho$ be a coupling of $\uppi (\mu )$ and $\uppi (\nu )$ concentrating on $\{ (e_0,e_1)\csuchthat e_0\prec _s e_1 \}$. Suppose toward a contradiction that $\rho ( \{ (e_0, e_1)\csuchthat e_0\succ _s e_1 \} ) < 1$. Then there exists an open subset $U$ of $\sP$ such that
\begin{equation}\label{eqn:BU1}
\rho ( \{ (e_0,e_1)\csuchthat e_0 \not\in B_U \mbox{ and }e_1\in B_U \} ) >0,
\end{equation}
where $B_U = \{ \lambda \in \sP \csuchthat SW(\lambda )\cap U \neq \emptyset \}$. The condition $\rho ( \{ (e_0,e_1)\csuchthat e_0\prec _s e_1 \} ) =1$ implies that
\begin{equation}\label{eqn:BU2}
\rho ( \{ (e_0 , e_1)\csuchthat e_0 \in B_U \} ) = \rho ( \{ (e_0, e_1)\csuchthat e_0 \in B_U\mbox{ and } e_1\in B_U \} )
\end{equation}
Using (\ref{eqn:BU1}) and (\ref{eqn:BU2}) we compute
\begin{align*}
\uppi (\mu )(B_U) &= \rho  ( \{ ( e_0 , e_1) \csuchthat e_0 \in B_U \} )\\
&= \rho  (\{ (e_0, e_1)\csuchthat e_0 \in B_U\mbox{ and }e_1 \in B_U \} ) \\
&< \rho  (\{ (e_0, e_1)\csuchthat e_0 \in B_U\mbox{ and }e_1 \in B_U \} ) + \rho  ( \{ (e_0,e_1)\csuchthat e_0 \not\in B_U \mbox{ and }e_1\in B_U \} ) \\
&= \rho  ( \{ (e_0, e_1)\csuchthat e_1\in B_U \} ) \\
&= \uppi (\nu )(B_U) .
\end{align*}
On the other hand, since $\mu \succ _s \nu$, part (i) implies that we can find coupling $\widetilde{\rho}$ of $\uppi (\mu )$ and $\uppi (\nu )$ such that $\widetilde{\rho} ( \{ (e_0, e_1)\csuchthat e_0\succ _s e_1 \} )=1$ and therefore
\[
\uppi (\nu )(B_U) = \widetilde{\rho} ( \{ (e_0 , e_1) \csuchthat e_1 \in B_U \} ) \leq \widetilde{\rho}  ( \{ (e_0,e_1)\csuchthat e_0 \in B_U \} ) = \uppi (\mu )(B_U) ,
\]
a contradiction.
\end{proof}

\section{Convexity}\label{sec:convexity}

The space $\rm{CloCon}(\sP)$, of all closed convex subsets of $\sP$, is naturally endowed with a convex structure: if $F_1,F_2 \in \rm{CloCon}(\sP)$ and $0\le t\le 1$ then
$$tF_1 + (1-t)F_2 := \{t\mu_1 + (1-t)\mu_2:~\mu_1 \in F_1, \mu_2\in F_2\}.$$

More generally, if $(\Omega,\omega)$ is a probability space and $F:\Omega \to \rm{CloCon}(\sP)$ a measurable map then
$$\int F(x)~d\omega(x) \subseteq \sP$$
denotes the set of all measures in $\sP$ of the form $\int \sigma(x)~d\omega(x)$ where $\sigma$  runs over all measurable
$\sigma:\Omega \to \sP$ satisfying $\sigma(x) \in F(x)$ for $\omega$-a.e. $x$.

\begin{thm}\label{thm:affine}
Let $\omega$ be a Borel probability measure on $\sP$ and assume that there is an $\omega$-conull set $\sP _{\omega}\subseteq \sP$ such that the measures in $\sP _{\omega}$  are mutually singular. Then
$$\int SW(\mu)~d\omega(\mu) = SW\left( \int \mu~d\omega(\mu)\right).$$
It follows that $\Es{SW}_\Ga$ is convex.
\end{thm}

\begin{remark}
Theorem \ref{thm:affine} implies that $SW(t\bfa \oplus (1-t)\bfb ) =tSW(\bfa)+(1-t)SW(\bfb)$ for all p.m.p.\ actions $\bfa$ and $\bfb$ of $\Ga$, and all $t\in [0,1]$. This is because we can find isomorphic copies $\mu _{\bfa}, \mu _{\bfb} \in \sP$ of $\bfa$ and $\bfb$ respectively, whose supports are disjoint, and hence by Theorem \ref{thm:affine}
\[
SW(t\bfa \oplus (1-t)\bfb ) = SW(t\mu _{\bfa} + (1-t)\mu _{\bfb}) =tSW(\mu _{\bfa}) + (1-t)SW(\mu _{\bfb}) = tSW(\bfa ) + (1-t)SW(\bfb ) .
\]
\end{remark}

\begin{proof}
By Lemma \ref{lem:semiaffine}, if $f:\sP \to \sP$ is a measurable map satisfying $f(\mu) \prec_s \mu$ for $\omega$-a.e. $\mu$ then
$$\int f(\mu)~d\omega(\mu) \in SW\left( \int \mu~d\omega(\mu)\right).$$
This proves $\int SW(\mu)~d\omega(\mu) \subseteq SW\left( \int \mu~d\omega(\mu)\right).$

To prove the opposite containment, suppose that $\nu \in SW\left( \int \mu~d\omega(\mu)\right)$. Then by Theorem \ref{thm:Scoupling} there exists a coupling $\rho$ of $\uppi(\nu)$ and $\uppi ( \int \mu \, d\omega )= \uppi (\upbeta (\omega  ))$ such that
$$\rho(\{ (e_1,e_2) \in \sP^{erg}\times \sP^{erg}:~e_1 \prec_s e_2\})=1.$$
Let $\rho = \int _{\sP ^{erg}}\rho ^e \times \updelta _e \, d\uppi (\upbeta (\omega ))$ be the disintegration of $\rho$ over $\uppi (\upbeta (\omega ))$. Then $\upbeta (\rho ^e ) \prec _s e$ for $\uppi (\upbeta (\omega ))$-almost every $e\in \sP ^{erg}$, so after redefining $\rho ^e$ on a $\uppi (\upbeta (\omega ))$-null set if necessary we may assume without loss of generality that $\upbeta (\rho ^e ) \prec _s e$ for all $e\in \sP ^{erg}$. For each $\mu \in \sP$ let $\rho ^{\mu} := \int \rho ^e  \, d\uppi (\mu ) (e)$. Then $\upbeta (\rho ^{\mu}) \prec _s \upbeta (\uppi (\mu )) = \mu$ by Lemma \ref{lem:semiaffine}. Since $\uppi (\upbeta (\omega )) = \int \uppi (\mu ) \, d\omega$ we have
\[
\int \rho ^{\mu} \, d\omega (\mu ) = \int \int \rho ^e \, d\uppi (\mu ) (e) \, d\omega (\mu ) = \int \rho ^e \, d(\uppi (\upbeta (\omega ))) (e) = \uppi (\nu ) .
\]
Therefore, $\mu \mapsto \upbeta (\rho ^{\mu})$ witnesses that $\nu \in \int SW(\mu)~d\omega(\mu)$. This proves that $\int SW(\mu)~d\omega(\mu) \supseteq SW\left( \int \mu~d\omega(\mu)\right).$

To see that $\Es{SW}_\Ga$ is convex, given $SW(\mu ), SW (\nu ) \in \Es{SW}_{\Ga}$ and $t\in [0,1]$, we can find isomorphic copies $\mu '$ and $\nu '$, of $\mu$ and $\nu$ respectively, whose supports are disjoint. Then $t SW(\mu ) + (1-t)SW(\nu ) =t SW(\mu ' ) + (1-t)SW(\nu ') = SW(t\mu ' + (1-t)\nu ' ) \in \Es{SW}_{\Ga}$. 
\end{proof}

\section{Simplex}

In this section, we prove $\Es{SW}_\Ga$ is a simplex. Let $\Es{SW}^{ext}_\Ga \subseteq \Es{SW}_\Ga$ denote the subspace of extreme stable weak equivalence classes. More  precisely, $S \in \Es{SW}^{ext}_\Ga$ if and only if the equation $S = t S_1 + (1-t)S_2$ with $S_1,S_2 \in \Es{SW}_\Ga$ and $t \in (0,1)$ implies $S_1=S_2=S$.

\begin{thm}\label{thm:simplex}
For each stable weak equivalence class $S\in \Es{SW}_\Ga$ there exists a unique Borel probability measure $\uppi (S)$ on $\Es{SW}^{ext}_\Ga$ such that $S = \int _{E\in \Es{SW}^{ext}_\Ga} E\ d\uppi (S)$. Furthermore, for any $\mu \in \sP$ we have $\uppi (SW(\mu )) = SW_*\uppi (\mu )$.
\end{thm}

\begin{lem}\label{lem:extreme}
If $S \in \Es{SW}_\Ga$ is a subsimplex of $\sP$ then it is extreme.  In particular, if $\mu \in \sP^{erg}$ then $SW(\mu)\in \Es{SW}_\Ga^{ext}$. Conversely, if $S \in \Es{SW}^{ext}_\Ga$ then there exist an ergodic $\mu \in \sP^{erg}$ such that $S = SW(\mu)$.
\end{lem}

\begin{proof}
Let $S \in \Es{SW}_\Ga$ be a subsimplex of $\sP$ and suppose $S = tS_1 + (1-t)S_2$ for some $S_1,S_2 \in \Es{SW}_\Ga$ and $t \in (0,1)$. For every ergodic measure $\nu \in S$ we must be able to write $\nu = t\nu_1 + (1-t) \nu_2$ for some $\nu_i \in S_i$ ($i=1,2$). Since $\nu$ is ergodic, $\nu_1=\nu_2=\nu$. So $S \cap \sP^{erg} \subseteq S_1 \cap S_2$. By hypothesis, $S$ is the closed convex hull of $S \cap \sP^{erg}$. Since $S_1$ and $S_2$ are convex, $S \subseteq S_1 \cap S_2$. To obtain a contradiction, suppose $\nu_1 \in S_1 \setminus S$. Let $\nu_2 \in S_2$. Then $t\nu_1 + (1-t)\nu_2 \in S$. By the ergodic decomposition theorem, almost every ergodic component of $\nu_1$ must be contained in $S$ and therefore, $\nu_1 \in S$. This contradiction shows that $S_1 \cap S_2 \subseteq S$. So $S=S_1=S_2$ as claimed.

Suppose $\mu \in \sP^{erg}$. By Theorem \ref{thm:stable}, $SW(\mu)$ is a subsimplex of $\sP$. So the previous paragraph implies $SW(\mu) \in \Es{SW}^{ext}_\Ga$.

For the converse, suppose $S \in \Es{SW}_\Ga^{ext}$. Let $\mu \in \sP$ such that $S=SW(\mu)$. By Theorem \ref{thm:affine},
$$S=SW(\mu) = \int SW(e)~d\uppi (\mu )(e).$$
Since $SW(\mu)$ is extreme, we must have $SW(e) = S$ for $\uppi(\mu)$-a.e. $e \in \sP^{erg}$.
\end{proof}

\begin{proof}[Proof of Theorem \ref{thm:simplex}]

Lemma \ref{lem:extreme} shows that $SW$ maps $\sP^{erg}$ onto $\Es{SW}^{ext}_\Ga$. So $SW_* : \Prob(\sP) \ra \Prob(\Es{SW}_\Ga)$ maps $\Prob(\sP^{erg})$ onto $\Prob(\Es{SW}^{ext}_\Ga)$. In addition, if $\mu \in \sP$ then $SW_*\uppi (\mu )$ is a Borel probability measure on $\Es{SW}^{ext}_\Ga$ whose barycenter is $SW(\mu )$ since
$$\int _{\Es{SW}^{ext}_\Ga} E \ dSW_*\uppi (\mu )(E) = \int  SW(e) \, d\uppi (\mu )(e) = SW\left(\int e \, d\uppi (\mu )(e) \right) =SW (\mu )$$
where the second equality holds by Theorem \ref{thm:affine} and the other equalities hold by definition. This shows that every stable weak equivalence class is represented by a measure on $\Es{SW}^{ext}_\Ga$. We now show that this representation is unique.

Let $\kappa _0$ and $\kappa _1$ be Borel probability measures on $\Es{SW}^{ext}_\Ga$ with $\int  E \, d\kappa _0(E) = S = \int  E \, d\kappa _1(E)$. We must show that $\kappa _0 = \kappa _1$. By \cite[Theorem 18.1]{Ke95} and Lemma \ref{lem:extreme} there exists a universally measurable map $s: \Es{SW}^{ext}_\Ga\ra \sP^{erg}$ with $SW(s(E))= E$ for all $E\in \Es{SW}^{ext}_\Ga$. For $i\in \{ 0,1 \}$ let $\mu _i = \upbeta (s_*\kappa _i) \in \sP$. Then
$$SW( \mu _i ) = \int  SW (e) \, ds_*\kappa _i(e) = \int SW(s(E))\, d\kappa _i(E) = \int  E\, d\kappa _i(E) = S,$$
so $\mu _0$ and $\mu _1$ are stably weakly equivalent. By Theorem \ref{thm:Scoupling} there exists a coupling $\rho$ of $\uppi (\mu _0)$ and $\uppi (\mu _1)$ with $\rho (\{ (e_0,e_1)\csuchthat e_0\sim _s e_1 \} ) =1$. We have $\uppi (\mu _i ) = \uppi (\upbeta (s_*\kappa _i )) = s_*\kappa _i$, so $\rho$ is a coupling of $s_*\kappa _0$ and $s_*\kappa _1$. Then $(SW\times SW)_*\rho$ is a coupling of $\kappa _0$ and $\kappa _1$ with
\[
(SW\times SW)_*\rho \left(\left\{ (E_0,E_1)\in (\Es{SW}^{ext}_\Ga )^2 \csuchthat E_0 = E_1 \right\} \right) = \rho (\{ (e_0,e_1)\csuchthat e_0\sim _s e_1 \} ) =1.
\]
It follows that for any Borel $B\subseteq \Es{SW}^{ext}_\Ga$ we have $\kappa _0 (B) = (SW\times SW)_*\rho (B\times \Es{SW}^{ext}_\Ga ) = (SW\times SW)_*\rho (B\times B) = (SW\times SW)_*\rho (\Es{SW}^{ext}_\Ga \times B ) = \kappa _1(B)$ and so $\kappa _0 = \kappa _1$.

The second statement follows from the first and the fact that $SW(\mu ) = \int  E\ dSW_*\uppi (\mu )(E)$.\qedhere
\end{proof}

In \cite[Theorem 1.5]{burton-weak-2015}, P. Burton shows that $\Es{SW}_\Ga$ is affinely homeomorphic to a convex compact subset of a Banach space. The proof uses an abstract characterization of convex compact subsets of Banach spaces due to Capraro and Fritz \cite{MR3034438}. It now follows from Theorem \ref{thm:simplex} that $\Es{SW}_\Ga$ is a Choquet simplex (equivalently, it is a convex compact subset of a locally convex topological vector space with the property that every element admits a unique representation as the barycenter of a probability measure on the space of extreme points).

\section{Property (T) groups}

\begin{thm}\label{thm:T}
Suppose $\Ga$ is a countable group with property (T). Then $\Es{SW}_\Ga$ is a Bauer simplex; the set $\Es{SW}^{ext}_\Ga \subseteq \Es{SW}_\Ga$ of extreme points is closed.
\end{thm}

\begin{proof}
Let $\{S_n\} \subseteq \Es{SW}^{ext}_\Ga$ be a sequence of extreme stable weak equivalence classes. Suppose $\lim_n S_n = S_\infty \in \Es{SW}_\Ga$. It suffices to show $S_\infty$ is extreme.

Because each $S_n$ is extreme, $S_n$ is a subsimplex of $\sP$ (Theorem \ref{thm:stable} and Lemma \ref{lem:extreme}). Therefore, it is the convex hull of $S_n \cap \sP^{erg}$. Because $\Ga$ has property (T), $\sP^{erg}$ is closed in $\sP$ \cite{glasner1997kazhdan}. After passing to a subsequence we may assume that $S_n \cap \sP^{erg}$ converges to some subset $K \subseteq \sP^{erg}$ as $n\to\infty$. But this implies $S_n$ converges to the convex hull of $K$; and therefore $S_\infty$ is the convex hull of $K$. So $S_\infty$ is a subsimplex of $\sP$ which implies that it is extreme by Lemma \ref{lem:extreme}.

\end{proof}

\section{Groups with many extreme stable weak equivalence classes}

In \cite{BG13}, Brown and Guentner associate a $C^*$-algebra $C^*_D(\Gamma )$ to each algebraic ideal $D$ in $\ell ^{\infty}(\Gamma )$. We will be concerned with the case $D=\ell ^p(\Gamma )$ for $2\leq p<\infty$, and we write $C^*_{\ell ^p} (\Gamma )$ for $C^*_{\ell ^p(\Gamma )}(\Gamma )$, which is defined as follows.

\begin{defn}[\cite{BG13}]
Let $\pi$ be a unitary representation of $\Gamma$ on a Hilbert space $\mc{H}_{\pi}$, and let $2\leq p < \infty$. The representation $\pi$ is said to be an {\bf $\ell ^p(\Gamma )$-representation} if there exists a dense linear subspace $\mc{H}_0$ of $\mathcal{H}_{\pi}$ such that for all $\xi ,\eta \in \mc{H}_0$ the matrix coefficient $\pi _{\xi , \eta} :\gamma \mapsto \langle \pi (\gamma  )\xi , \eta \rangle$, belongs to $\ell ^p(\Gamma )$.  The $C^*$-algebra $C^*_{\ell ^p} (\Gamma )$ is defined as the completion of the group ring $\C [\Gamma ]$ with respect to the $C^*$-norm
\[
\| x \| _{C^*_{\ell ^p}} := \sup \{ \| \pi (x) \| \, : \, \pi \text{ is an }\ell ^p(\Gamma  )\text{-representation} \} ,
\]
where $\| \pi (x) \|$ denotes the operator norm of $\pi (x)$.
\end{defn}

Since $\Gamma$ is countable, and since the direct sum of $\ell ^p(\Gamma )$-representations is an $\ell ^p(\Gamma )$-representation, we can in fact find an $\ell ^p(\Gamma )$-representation, denoted $\sigma ^p_{\Gamma}$, on a separable Hilbert space $\mathcal{H}_{\sigma ^p_{\Gamma}}$, such that $\| x \| _{C^*_{\ell ^p}} = \| \sigma ^p_{\Gamma }(x) \|$ for all $x\in \C [\Gamma ]$. Hence, $C^*_{\ell ^p}(\Gamma )$ is isomorphic to the $C^*$-subalgebra of $\mathcal{B}(\mathcal{H}_{\sigma ^p_{\Gamma}})$ generated by $\sigma ^p_{\Gamma}(\Gamma )$. By \cite[Chapter 18]{dixmier1977c} $\sigma ^p_{\Gamma}$ is uniquely defined up to weak equivalence of unitary representations, and, up to weak equivalence $\sigma ^p_{\Gamma}$ is the unique $\ell ^p(\Gamma )$-representation which weakly contains all other $\ell ^p(\Gamma )$-representations. If $p\leq q$ then $\| x \| _{C^*_{\ell ^p}}\leq \| x \| _{C^*_{\ell ^q}}$ for all $x\in \C [\Gamma ]$, and the canonical quotient map from $C^*_{\ell ^q}(\Gamma )$ onto $C^*_{\ell ^p}(\Gamma )$ is an isomorphism if and only if $\sigma ^p_{\Gamma}$ and $\sigma ^q_{\Gamma}$ are weakly equivalent \cite{dixmier1977c}. The main result of this section is a direct consequence of the following striking result of Okayasu.

\begin{thm}[\cite{Oka14}]\label{thm:Oka14}
Let $F_2$ denote the free group on two generators and let $2\leq p < q <\infty$. Then the canonical quotient map $C^*_{\ell ^q}(\Gamma )\rightarrow C^*_{\ell ^p}(\Gamma )$ is not injective, and hence the unitary representations $\sigma ^p_{F_2}$ and $\sigma ^q_{F_2}$ are weakly inequivalent.
\end{thm}

As observed in \cite{wiersma2016constructions}, since restrictions (and, respectively, inductions) of $\ell ^p$-representations to (respectively: from) subgroups are themselves $\ell ^p$-representations, it follows immediately from Theorem \ref{thm:Oka14} that if $\Gamma$ is any group containing a subgroup isomorphic to $F_2$, then the unitary representations $\sigma ^p_{\Gamma}$, $2\leq p < \infty$, are pairwise weakly inequivalent.

For each unitary representation $\pi$ of $\Gamma$ on a separable Hilbert space we consider the corresponding {\bf Gaussian action}, denoted $\bfa (\pi )$, which is a p.m.p.\ action of $\Gamma$ on a standard probability space (see \cite[Appendix E]{Kechris-global-aspects} and \cite[Appendix E]{MR3616077}). We let $\kappa ^{\bfa (\pi )}$ denote the Koopman representation corresponding to $\bfa (\pi )$, and we let $\kappa ^{\bfa (\pi )}_0$ denote the restriction  of $\kappa ^{\bfa (\pi )}$ to the orthogonal complement of the constant functions. We note the following lemma:

\begin{lem}\label{lem:lp}
The representations $\kappa ^{\bfa (\sigma ^p_{\Gamma}) }_0$ and $\sigma ^p_{\Gamma}$ are weakly equivalent.
\end{lem}

\begin{proof}
Put $\sigma = \sigma ^p_{\Gamma}$. By \cite[Theorem E.19]{MR3616077}, $\kappa ^{\bfa (\sigma )}_0$ contains $\sigma$ and is isomorphic to a subrepresentation of $\bigoplus _{n\geq 1}(\sigma \oplus \overline{\sigma })^{\otimes n}$, where  $\overline{\sigma }$ denotes the conjugate representation of $\sigma$. By \cite{BG13}, the representation $\bigoplus _{n\geq 1}(\sigma \oplus \overline{\sigma })^{\otimes n}$ is an $\ell ^p(\Gamma )$-representation, so (since it contains $\sigma$) it is weakly equivalent to $\sigma$. Therefore, $\kappa ^{\bfa (\sigma )}_0$ is weakly equivalent to $\sigma$ as well.
\end{proof}

By \cite{dixmier1977c}, every $\ell ^2 (\Gamma )$-representation is a subrepresentation of a multiple of the left regular representation of $\Gamma$. For concreteness, we will therefore take $\sigma ^2_{\Gamma}$ to be the left regular representation of $\Gamma$. Also, for each $2\leq p<\infty$, since $\ell ^2(\Gamma )\leq \ell ^p(\Gamma )$, we will assume (without loss of generality) that $\sigma  ^2_{\Gamma}$ is a subrepresentation of $\sigma ^p_{\Gamma}$. Then $\bfa (\sigma ^2_{\Gamma})$ is a Bernoulli shift action of $\Gamma$, and for each $2\leq p<\infty$ the action $\bfa (\sigma ^p_{\Gamma})$ factors onto a Bernoulli shift and hence is free.

\begin{thm}\label{thm:uncountable}
Let $\Gamma$ be a group containing a subgroup isomorphic to $F_2$. Then the actions $\bfa (\sigma ^p_{\Gamma} )$, $2\leq p < \infty$, are pairwise stably weakly inequivalent, and each is free, mixing and strongly ergodic.
\end{thm}

\begin{proof}
We already observed that each of the actions $\bfa (\sigma ^p_{\Gamma} )$ is free. Since $\Gamma$ is non-amenable, the representation $\sigma ^p_{\Gamma}$ does not have almost invariant vectors \cite{BG13}. Therefore, the representation $\kappa ^{\bfa (\sigma ^p _{\Gamma})}_0$, being weakly equivalent to $\sigma ^p_{\Gamma}$, does not have almost invariant vectors.  This implies that $\bfa (\sigma ^p _{\Gamma})$ is strongly ergodic. Since $\ell ^p(\Gamma ) \subseteq c_0 (\Gamma )$, each of the representations $\sigma ^p_{\Gamma}$ is mixing, hence the action $\bfa (\sigma ^p_{\Gamma} )$ is mixing.

If $\bfa (\sigma ^p_{\Gamma } ) \sim _s \bfa (\sigma ^q_{\Gamma} )$, then $\bfa (\sigma ^p_{\Gamma } ) \sim \bfa (\sigma ^q_{\Gamma} )$ since both actions are ergodic, and hence $\kappa ^{\bfa (\sigma ^p_{\Gamma})}_0 \sim \kappa ^{\bfa (\sigma ^q_{\Gamma})}_0$. Lemma \ref{lem:lp} then implies that $\sigma ^p _{\Gamma}\sim \kappa ^{\bfa (\sigma ^p_{\Gamma } ) }_0 \sim \kappa ^{\bfa (\sigma ^q_{\Gamma } ) }_0\sim \sigma ^q_{\Gamma}$, and so we must have $p=q$ by the remark following Theorem \ref{thm:Oka14}.
\end{proof}

\bibliography{biblio}
\bibliographystyle{alpha}

\end{document}